\documentclass[10pt]{article}

\usepackage{graphicx}
\usepackage{subcaption}

\usepackage[a4paper, left=35mm,right=35mm,top=34mm,bottom=34mm]{geometry}
\usepackage[utf8]{inputenc}
\usepackage[T1]{fontenc}
\usepackage[english]{babel}

\usepackage{enumerate}
\usepackage{graphicx}
\usepackage{hyperref}
\hypersetup{
    colorlinks=true,
    linkcolor=blue,
    filecolor=magenta,      
    urlcolor=cyan,
}

\usepackage{mathtools,amsthm,amssymb,amsfonts}
\usepackage{algorithm}
\usepackage{algpseudocode}
\makeatletter
\def\thm@space@setup{
  \thm@preskip=10pt \thm@postskip=10pt
}
\makeatother
\theoremstyle{plain}
\newtheorem{theorem}{Theorem}
\theoremstyle{plain}
\newtheorem{lemma}[theorem]{Lemma}
\theoremstyle{definition}

\theoremstyle{definition}
\newtheorem{example}{Example}
\theoremstyle{remark}
\newtheorem*{remark}{Remark}
\theoremstyle{remark}

\usepackage{caption} 
\usepackage{booktabs}
\usepackage{listings}
\usepackage{color}
\definecolor{dkgreen}{rgb}{0,0.6,0}
\definecolor{gray}{rgb}{0.5,0.5,0.5}
\definecolor{mauve}{rgb}{0.58,0,0.82}
\usepackage{enumitem}

\newcommand{\abs}[1]{\left\lvert#1\right\lvert}
\newcommand{\norm}[1]{\left\lVert#1\right\rVert}

\DeclareMathOperator*{\argmin}{arg\,min}

\newcommand{\email}[1]{\protect\href{mailto:#1}{#1}}

\usepackage{xcolor}[2.11]
\colorlet{inlinkcolor}{green!50!black}
\colorlet{exlinkcolor}{red!50!black}
\if@hidelinks
\hypersetup{ hidelinks = true }
\else
\hypersetup{
  colorlinks = true,
  allcolors = inlinkcolor,
  urlcolor = exlinkcolor,
}
\fi

\newenvironment{@abssec}[1]{
        \vspace{.05in}\parindent .0in
        {\upshape\bfseries #1. }\ignorespaces
    }
    {\par\vspace{.1in}}
\renewenvironment{abstract}{\begin{@abssec}{\abstractname}}{\end{@abssec}}
\newenvironment{keywords}{\begin{@abssec}{Keywords}}{\end{@abssec}}

\usepackage{fancyhdr}

\author{
  {\normalsize Qinmeng Zou}\thanks{CentraleSup\'elec, Universit\'e Paris-Saclay, 3 rue Joliot Curie, 91190 Gif-sur-Yvette, France
    (\email{zouqinmeng@gmail.com}, \email{frederic.magoules@hotmail.com}).}
  \and
  {\normalsize Fr\'ed\'eric Magoul\`es\footnotemark[1]}
}
\title{Parameter Estimation in the Hermitian and Skew-Hermitian Splitting Method Using Gradient Iterations}
\date{}

\begin{document}
\maketitle
\thispagestyle{fancy}

\begin{abstract}
This paper presents enhancement strategies for the Hermitian and skew-Hermitian splitting method based on gradient iterations.
The spectral properties are exploited for the parameter estimation, often resulting in a better convergence.
In particular, steepest descent with early stopping can generate a rough estimate of the optimal parameter.
This is better than an arbitrary choice since the latter often causes stability problems or slow convergence.
Additionally, lagged gradient methods are considered as inner solvers for the splitting method.
Experiments show that they are competitive with conjugate gradient in low precision.
\end{abstract}

\begin{keywords}
Hermitian and skew-Hermitian splitting; steepest descent; minimal gradient; parameter estimation; lagged gradient methods; Barzilai-Borwein method.
\end{keywords}

\section{Introduction}
\label{sec:1}

We are interested in solving the linear system
\begin{equation}
\label{eq:ls}
Ax=b,
\end{equation}
where $A$ is a non-Hermitian positive definite matrix of size $N$.
It has been observed that splitting methods can be used with success.
The traditional alternating direction implicit method~\cite{Peaceman1955} has inspired the construction of alternate two-step splittings $A=\mathcal{M}_1-\mathcal{N}_1$ and $A=\mathcal{M}_2-\mathcal{N}_2$, and this leads to an iteration called Hermitian and skew-Hermitian splitting (HSS)~\cite{Bai2003} in which alternately a shifted Hermitian system and a shifted skew-Hermitian system are solved.
HSS has received so much attention~\cite{Bai2007b,Bai2008,Benzi2009,Salkuyeh2015,Wu2015,Wu2017,Pourbagher2018}, possibly due to its guaranteed convergence and mathematical beauty.

Let $H$ and $S$ denote the Hermitian and skew-Hermitian parts of $A$, respectively.
Let $A^\mathsf{H}$ be the conjugate transpose of matrix $A$.
It follows that
\[
H = \frac{A+A^\mathsf{H}}{2},\quad S = \frac{A-A^\mathsf{H}}{2}.
\]
Let $I$ be the identity matrix.
In short, the HSS method is defined as follows
\begin{equation}
\label{eq:hss}
\left\{\begin{array}{l}
(\gamma I + H)x_{n+\frac{1}{2}} = (\gamma I - S)x_n + b, \\[1ex]
(\gamma I + S)x_{n+1} = (\gamma I - H)x_{n+\frac{1}{2}} + b,
\end{array}\right.
\end{equation}
with $\gamma>0$.
It could be regarded as a stationary iterative process
\[
x_{n+1} = Tx_n + p,
\]
where $x_0$ is a given vector.
Following the notations of Bai et al.~\cite{Bai2003} let us set
\[
\mathcal{M}_1 = \gamma I + H,\quad \mathcal{N}_1 = \gamma I - S,\quad \mathcal{M}_2 = \gamma I + S,\quad \mathcal{N}_2 = \gamma I - H.
\]
The operators $T$ and $p$ can be expressed as
\[
T = \mathcal{M}_2^{-1}\mathcal{N}_2\mathcal{M}_1^{-1}\mathcal{N}_1,\quad p = \mathcal{M}_2^{-1}(I+\mathcal{N}_2\mathcal{M}_1^{-1})b.
\]
Let $\sigma(\cdot)$ be the spectrum of a matrix and let $\rho(\cdot)$ be the spectral radius.
Convergence result for~\eqref{eq:hss} in the non-Hermitian positive definite case was established by Bai et al.~\cite{Bai2003}
\begin{equation}
\label{eq:bd}
\rho(T) \le \norm{\mathcal{N}_2\mathcal{M}_1^{-1}} = \max_{\lambda\in\sigma(H)} \frac{\abs{\lambda-\gamma}}{\abs{\lambda+\gamma}},
\end{equation}
where $\norm{\cdot}$ denotes $2$-norm.
This shows that the spectral radius of iteration matrix $T$ is less than $1$.
As a result, HSS has guaranteed convergence for which the speed depends only on the Hermitian part $H$.
Let $\lambda_i(\cdot)$ be the $i$th eigenvalue of a matrix in ascending order.
The key observation here is that choosing
\begin{equation}
\label{eq:gamma}
\gamma = \gamma_* = \sqrt{\lambda_1(H)\lambda_N(H)}
\end{equation}
leads to the well-known upper bound
\begin{equation}
\label{eq:bd:op}
\rho(T) \le \frac{\sqrt{\kappa(H)}-1}{\sqrt{\kappa(H)}+1},
\end{equation}
where $\kappa(\cdot)$ denotes the condition number.
It is noteworthy that inequality~\eqref{eq:bd:op} is similar to the convergence result for conjugate gradient (CG)~\cite{Hestenes1952,vanderSluis1986} in terms of $A$-norm error.
As mentioned by Bai et al.~\cite{Bai2003}, $\gamma_*$ minimizes the upper bound of $\rho(T)$ but not $\rho(T)$ itself.
In some cases the right-hand side of~\eqref{eq:bd} may not be an accurate approximation to the spectral radius.
Since very little theory is available on direct minimization, we still try to approximate indirectly the optimal parameter~$\gamma_*$.

In this paper we exploit spectral properties of gradient iterations in order to make the estimation feasible.
In Section~\ref{sec:2}, we focus on the asymptotic analysis of the steepest descent method.
In Section~\ref{sec:3}, we discuss some strategies for estimating the parameter in HSS based on gradient iterations and give a comparison of lagged gradient methods and CG for solving Hermitian positive definite systems in low precision.
Numerical results are shown in Section~\ref{sec:4} and some concluding remarks are drawn in Section~\ref{sec:5}.

\section{Asymptotic analysis of steepest descent}
\label{sec:2}

In this section we consider the Hermitian positive definite (HPD) linear system
\begin{equation}
\label{eq:ls:hpd}
Hx=\hat{b}
\end{equation}
of size $N$.
The solution $x_*$ is the unique global minimizer of convex quadratic function
\begin{equation}
\label{eq:quad}
f(x) = \frac{1}{2}x^\mathsf{H} Hx - \hat{b}^\mathsf{H} x.
\end{equation}
For $n = 0,\,1,\,\dots$, the gradient method is of the form
\begin{equation}
\label{eq:x}
x_{n+1} = x_n - \alpha_n g_n,
\end{equation}
where $g_n = \nabla f(x_n) = Hx_n-\hat{b}$.
This gives the updating formula
\begin{equation}
\label{eq:g}
g_{n+1} = g_n - \alpha_n Hg_n.
\end{equation}
The steepest descent (SD) method proposed by Cauchy~\cite{Cauchy1847} defines a sequence of steplengths as follows
\begin{equation}
\label{eq:sd}
\alpha_n^\text{SD} = \frac{g_n^\mathsf{H} g_n}{g_n^\mathsf{H} Hg_n},
\end{equation}
which is the reciprocal of Rayleigh quotient.
It minimizes the quadratic function $f$ or the $A$-norm error of the system~\eqref{eq:ls:hpd} and gives theoretically an optimal result at each step
\[
\alpha_n^\text{SD} = \argmin_{\alpha}f(x_n-\alpha g_n) = \argmin_{\alpha}\norm{(I-\alpha H)e_n}_H^2,
\]
where $e_n=x_*-x_n$.
This classical method is known to behave badly in practice.
The directions tend to asymptotically alternate between two orthogonal directions resulting in a slow convergence~\cite{Akaike1959}.

The motivation for this paper arose during the development of efficient gradient methods.
We notice that generally SD converges much slower than CG for HPD systems.
However, the spectral properties of the former could be beneficial to parameter estimation.
Akaike~\cite{Akaike1959} provided a probability distribution model for the asymptotic analysis of SD.
It appears that standard techniques used in linear algebra are not very helpful in this case.
The so-called two-step invariance property led to the work of Nocedal et al.~\cite{Nocedal2002} in which further asymptotic results are presented.
Let $v_i(\cdot)$ be the eigenvector corresponding to the eigenvalue~$\lambda_i(\cdot)$.
Relevant properties by Nocedal et al.~\cite{Nocedal2002} which will be exploited in the following text can be briefly described in Lemma~\ref{thm:1}.
Note that a symmetric positive definite real matrix was used by Nocedal et al.~\cite{Nocedal2002}.
Therefore, we extend this result and we present a new lemma and its proof in the case of Hermitian positive definite systems.
\begin{lemma}
\label{thm:1}
Assume that $\lambda_1(H) < \dots < \lambda_N(H)$.
Assume that $v_1^\mathsf{H}(H) g_0 \ne 0$ and $v_N^\mathsf{H}(H) g_0 \ne 0$.
Consider the gradient method~\eqref{eq:x} with steplength~\eqref{eq:sd} being used to solve~\eqref{eq:ls:hpd} where $H$ est Hermitian positive definite.
Then
\begin{equation}
\label{eq:1.1}
\lim_{n\rightarrow\infty}\alpha_{2n}^\text{SD} = \frac{1+c^2}{\lambda_1(H)(1+c^2\kappa(H))},
\end{equation}
\begin{equation}
\label{eq:1.2}
\lim_{n\rightarrow\infty}\alpha_{2n+1}^\text{SD} = \frac{1+c^2}{\lambda_1(H)(c^2+\kappa(H))},
\end{equation}
and
\begin{equation}
\label{eq:1.3}
\lim_{n\rightarrow\infty}\frac{\norm{g_{2n+1}}^2}{\norm{g_{2n}}^2} = \frac{c^2(\kappa(H)-1)^2}{(1+c^2\kappa(H))^2},
\end{equation}
\begin{equation}
\label{eq:1.4}
\lim_{n\rightarrow\infty}\frac{\norm{g_{2n+2}}^2}{\norm{g_{2n+1}}^2} = \frac{c^2(\kappa(H)-1)^2}{(c^2+\kappa(H))^2},
\end{equation}
for some constant $c$.
\end{lemma}
\begin{proof}
Let Re$(\cdot)$ and Im$(\cdot)$ be the real and imaginary parts, respectively.
The coefficients of system~\eqref{eq:ls:hpd} have the following form
\[
H = \text{Re}(H)+\iota\text{Im}(H),\quad x = \text{Re}(x)+\iota\text{Im}(x),\quad \hat{b} = \text{Re}(\hat{b})+\iota\text{Im}(\hat{b}),
\]
where $\iota$ denotes the imaginary unit.
It is possible to rewrite system~\eqref{eq:ls:hpd} into the real equivalent form
\begin{equation}
\label{eq:pf:1.1}
\tilde{H}\tilde{x} =
\left(\begin{array}{cc}
\text{Re}(H) & -\text{Im}(H) \\[1ex]
\text{Im}(H) & \text{Re}(H)
\end{array}\right)
\left(\begin{array}{c}
\text{Re}(x) \\[1ex]
\text{Im}(x)
\end{array}\right)
=
\left(\begin{array}{c}
\text{Re}(\hat{b}) \\[1ex]
\text{Im}(\hat{b})
\end{array}\right)
= \tilde{b}.
\end{equation}
By Lemma~3.3 and Theorem~5.1 in Nocedal et al., 2002~\cite{Nocedal2002}, it is known that results~\eqref{eq:1.1} to~\eqref{eq:1.4} hold in the real case.
To prove the desired result in the Hermitian case, it suffices to show that SD applied to~\eqref{eq:pf:1.1} is equivalent to that for~\eqref{eq:ls:hpd}, namely, they should yield the same sequences of gradient vectors and steplengths.
One finds that
\begin{equation}
\label{eq:pf:1.2}
g_n = (\text{Re}(H)+\iota\text{Im}(H))(\text{Re}(x_n)+\iota\text{Im}(x_n)) - (\text{Re}(\hat{b})+\iota\text{Im}(\hat{b})) = \varphi_n+\iota\psi_n,
\end{equation}
where
\[
\begin{split}
\varphi_n &= \text{Re}(H)\text{Re}(x_n)-\text{Im}(H)\text{Im}(x_n)-\text{Re}(\hat{b}), \\[1ex]
\psi_n &= \text{Re}(H)\text{Im}(x_n)+\text{Im}(H)\text{Re}(x_n)-\text{Im}(\hat{b}).
\end{split}
\]
Assume that the $2$ blocks in~$\tilde{x}_n$ is the same as the real and imaginary parts of~$x_n$, respectively.
Then, from~\eqref{eq:pf:1.1} one obtains that
\begin{equation}
\label{eq:pf:1.3}
\tilde{g}_n = \tilde{H}\tilde{x}_n-\tilde{b} =
\left(\begin{array}{c}
\varphi_n \\[1ex]
\psi_n
\end{array}\right).
\end{equation}
On the other hand, let
\[
\tilde{\alpha}_n=\frac{\tilde{g}_n^\intercal\tilde{g}_n}{\tilde{g}_n^\intercal\tilde{H}\tilde{g}_n}.
\]
Combining~\eqref{eq:pf:1.2} and~\eqref{eq:pf:1.3} implies $g_n^\mathsf{H}g_n=\tilde{g}_n^\intercal\tilde{g}_n$.
Since $\text{Im}(H)^\intercal=-\text{Im}(H)$, it follows that $u^\intercal\text{Im}(H)u=0$ for all $u\in\mathbb{R}^N$, from which one obtains that
\[
\text{Re}(g_n)^\intercal\text{Im}(H)\text{Re}(g_n)=0,\quad \text{Im}(g_n)^\intercal\text{Im}(H)\text{Im}(g_n)=0.
\]
Hence, the following result holds:
\[
\begin{split}
g_n^\mathsf{H} Hg_n &= (\text{Re}(g_n)+\iota\text{Im}(g_n))^\mathsf{H}(\text{Re}(H)+\iota\text{Im}(H))(\text{Re}(g_n)+\iota\text{Im}(g_n)) \\[1ex]
&= \text{Re}(g_n)^\intercal\text{Re}(H)\text{Re}(g_n) + \text{Im}(g_n)^\intercal\text{Re}(H)\text{Im}(g_n)+2\text{Im}(g_n)^\intercal\text{Im}(H)\text{Re}(g_n)
\end{split}
\]
Along with~\eqref{eq:pf:1.1}, this implies that $g_n^\mathsf{H}Hg_n=\tilde{g}_n^\intercal\tilde{H}\tilde{g}_n$, according to which one finds that $\tilde{\alpha}_n=\alpha_n$ when the $2$ blocks in~$\tilde{g}_n$ are equal to the real and imaginary parts of~$g_n$, respectively.
Hence, the SD iteration for Hermitian system~\eqref{eq:ls:hpd} and that for $2$-by-$2$ real form yield exactly the same sequence of solutions.
Since properties~\eqref{eq:1.1} to~\eqref{eq:1.4} in the real case has been proved by Nocedal et al.~\cite{Nocedal2002}, we arrive at the desired conclusion.
\end{proof}

Concerning the assumption used in Lemma~\ref{thm:1}, if there exist repeated eigenvalues, then we can choose the eigenvectors so that the corresponding gradient components vanish~\cite{Fletcher2005}.
If $v_1^\mathsf{H}(H) g_0=0$ or $v_N^\mathsf{H}(H) g_0=0$, then the second condition can be replaced by inner eigenvectors with no effect on the theoretical results.

It took some time before the spectral properties described by Nocedal et al.~\cite{Nocedal2002} were applied for solving linear systems.
De~Asmundis et al.~\cite{DeAsmundis2013} proposed an auxiliary steplength
\begin{equation}
\label{eq:a}
\alpha_n^\text{A} = \left(\frac{1}{\alpha_{n-1}^\text{SD}}+\frac{1}{\alpha_n^\text{SD}}\right)^{-1},
\end{equation}
which could be used for efficient implementations of gradient methods.
The major result is a direct consequence of~\eqref{eq:1.1} and~\eqref{eq:1.2}.
We state the lemma without proof, see De~Asmundis et al., 2013~\cite{DeAsmundis2013} for further discussion.
\begin{lemma}
\label{thm:2}
Under the assumptions of Lemma~\ref{thm:1}, the following result holds
\begin{equation}
\label{eq:2.1}
\lim_{n\rightarrow\infty}\alpha_n^\text{A} = \frac{1}{\lambda_1(H)+\lambda_N(H)}.
\end{equation}
\end{lemma}

Another direction of approach was based on a delicate derivation by Yuan~\cite{Yuan2006}.
Let us write $\alpha_n^\text{RA}=\left(\alpha_n^\text{A}\right)^{-1}$ and
\begin{equation}
\label{eq:Gamma}
\Gamma_n = \frac{1}{\alpha_{n-1}^\text{SD}\alpha_n^\text{SD}}-\frac{\norm{g_n}^2}{\left(\alpha_{n-1}^\text{SD}\right)^2 \norm{g_{n-1}}^2}.
\end{equation}
Yuan~\cite{Yuan2006} developed a new auxiliary steplength of the form
\begin{equation}
\label{eq:Y}
\alpha_n^\text{Y} = \frac{2}{\alpha_n^\text{RA}+\sqrt{\left(\alpha_n^\text{RA}\right)^2 - 4\Gamma_n}}.
\end{equation}
which leads to some $2$-dimensional finite termination methods for solving system~\eqref{eq:ls:hpd}~\cite{Yuan2006}.
Let us now introduce an alternative steplength
\begin{equation}
\label{eq:Z}
\alpha_n^\text{Z} = \frac{2}{\alpha_n^\text{RA}-\sqrt{\left(\alpha_n^\text{RA}\right)^2 - 4\Gamma_n}}.
\end{equation}
Let us write $\alpha_n^\text{RY}=\left(\alpha_n^\text{Y}\right)^{-1}$ and $\alpha_n^\text{RZ}=\left(\alpha_n^\text{Z}\right)^{-1}$.
It follows that
\[
\alpha_n^\text{RY}+\alpha_n^\text{RZ} = \alpha_n^\text{RA},\quad \alpha_n^\text{RY}\alpha_n^\text{RZ} = \Gamma_n.
\]
The spectral properties of~\eqref{eq:Gamma}, \eqref{eq:Y} and~\eqref{eq:Z} are shown in Lemma~\ref{thm:3}.
Note that the equations~\eqref{eq:3.1} and~\eqref{eq:3.2} have appeared in De~Asmundis et al., 2014~\cite{DeAsmundis2014} for the real case.
Below, we extend equations~\eqref{eq:3.1} and~\eqref{eq:3.2} for the Hermitian case and also one new equation.
\begin{lemma}
\label{thm:3}
Under the assumptions of Lemma~\ref{thm:1}, the following limits hold
\begin{equation}
\label{eq:3.1}
\lim_{n\rightarrow\infty}\Gamma_n = \lambda_1(H)\lambda_N(H).
\end{equation}
\begin{equation}
\label{eq:3.2}
\lim_{n\rightarrow\infty}\alpha_n^\text{Y} = \frac{1}{\lambda_N(H)}.
\end{equation}
\begin{equation}
\label{eq:3.3}
\lim_{n\rightarrow\infty}\alpha_n^\text{Z} = \frac{1}{\lambda_1(H)}.
\end{equation}
\end{lemma}
\begin{proof}
Combining~\eqref{eq:1.1} and~\eqref{eq:1.2} implies
\begin{equation}
\label{eq:pf:3.1}
\lim_{n\rightarrow\infty}\alpha_{n-1}^\text{SD}\alpha_n^\text{SD} = \frac{(1+c^2)^2}{\lambda_1^2(H)(c^2+\kappa(H))(1+c^2\kappa(H))}.
\end{equation}
Combining~\eqref{eq:1.1} to~\eqref{eq:1.4}, one could deduce that
\[
\lim_{n\rightarrow\infty}\frac{1}{(\alpha_{2n}^\text{SD})^2}\lim_{n\rightarrow\infty}\frac{\norm{g_{2n+1}}^2}{\norm{g_{2n}}^2} = \lim_{n\rightarrow\infty}\frac{1}{(\alpha_{2n+1}^\text{SD})^2}\lim_{n\rightarrow\infty}\frac{\norm{g_{2n+2}}^2}{\norm{g_{2n+1}}^2},
\]
from which one finds
\begin{equation}
\label{eq:pf:3.2}
\lim_{n\rightarrow\infty}\frac{\norm{g_n}^2}{\left(\alpha_{n-1}^\text{SD}\right)^2 \norm{g_{n-1}}^2} = \frac{\lambda_1^2(H)c^2(\kappa(H)-1)^2}{(1+c^2)^2}.
\end{equation}
The first equation follows by combining~\eqref{eq:pf:3.1} and~\eqref{eq:pf:3.2}.
Along with~\eqref{eq:2.1}, this implies that
\[
\lim_{n\rightarrow\infty}\left(\left(\alpha_n^\text{RA}\right)^2 - 4\Gamma_n\right) = \left(\lambda_1(H)-\lambda_N(H)\right)^2,
\]
which yields the desired limits~\eqref{eq:3.2} and~\eqref{eq:3.3}.
\end{proof}

It is noteworthy that steplengths~\eqref{eq:Y} and~\eqref{eq:Z} could be expressed as the roots of a quadratic function
\begin{equation}
\label{eq:Q}
Q_n(\alpha) = \Gamma_n\alpha^2-\alpha_n^\text{RA}\alpha+1,
\end{equation}
with
\[
Q_n(0)=1,\quad Q_n(\alpha_n^\text{A})=\Gamma_n\left(\alpha_n^\text{A}\right)^2,
\]
\[
Q_n(\alpha_{n-1}^\text{SD})=-\frac{\norm{g_n}^2}{\norm{g_{n-1}}^2},\quad Q_n(\alpha_n^\text{SD})=-\frac{\left(\alpha_n^\text{SD}\right)^2\norm{g_n}^2}{\left(\alpha_{n-1}^\text{SD}\right)^2\norm{g_{n-1}}^2},
\]
from which one could observe that $\Gamma_n>0$ and
\[
\alpha_n^\text{A} < \alpha_n^\text{Y}<\min\{\alpha_{n-1}^\text{SD},\,\alpha_n^\text{SD}\}.
\]
As mentioned by Yuan~\cite{Yuan2006}, a slightly shortened steplength would improve the efficiency of steepest descent.
This is one reason why the Yuan steplength could be fruitfully used in alternate gradient methods~\cite{Dai2005c,DeAsmundis2014}.

As an example, assume that $x_0=0$ and
\begin{equation}
\label{eq:d8}
H = \text{diag}(1,\,2,\,10,\,20,\,100,\,200,\,1000,\,2000).
\end{equation}
Assume that $\hat{b}$ is constructed by $\hat{b}=Hx_*$ where $x_*$ is a vector of all ones.
We plot in Figure~\ref{fig:1} the curves of~\eqref{eq:Q} for a few representative iteration numbers.
\begin{figure}[t]
\centerline{\includegraphics[width=.7\linewidth]{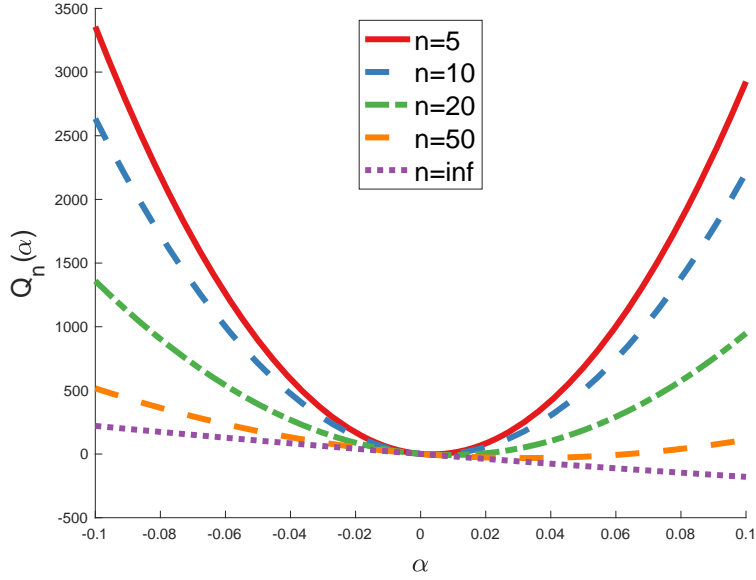}}
\caption{Curves of~$Q_n(\alpha)$ for a few representative iteration numbers. Steepest descent is used for solving system~\eqref{eq:ls:hpd} where $H$ satisfies~\eqref{eq:d8} and $\hat{b}$ is a vector of all ones.\label{fig:1}}
\end{figure}
This figure shows that the curves of $Q_n(\alpha)$ corresponding to steepest descent converge to the limit, as proved in Lemma~\ref{thm:2} and Lemma~\ref{thm:3}.

\section{Application to HSS iterations}
\label{sec:3}

\subsection{Preliminary considerations}
\label{sec:3.1}

In this section we first try to compute estimates for parameter~$\gamma$ in the HSS method.
One possible solution is to simply choose $\gamma=1$ without resorting to special techniques, but experience shows that it often leads to very slow convergence or even divergence, depending on the system being solved.
Another approach is based on the observation that $\gamma$ was introduced to enable the bounded convergence, as seen in~\eqref{eq:bd}, and it is possible to express it differently.
As an example consider a positive definite diagonal matrix $D$ such that
\begin{equation}
\label{eq:dhss}
\left\{\begin{array}{l}
(D + H)x_{n+\frac{1}{2}} = (D - S)x_n + b, \\[1ex]
(D + S)x_{n+1} = (D - H)x_{n+\frac{1}{2}} + b.
\end{array}\right.
\end{equation}
As a result, the iteration matrix is of the form
\[
T_D = (D+S)^{-1}(D-H)(D+H)^{-1}(D-S).
\]
Notice that~\eqref{eq:dhss} is a special case of preconditioned HSS~\cite{Bertaccini2005} when choosing $\gamma=1$ and $P=D$.
In particular, the fact that Theorem~2.1 in Bertaccini et al., 2005~\cite{Bertaccini2005} holds for~\eqref{eq:dhss} implies $\rho(T_D)<1$, yielding the guaranteed convergence.

On the basis of similar reasoning as in HSS~\cite{Bai2003}, the spectral radius is bounded by
\[
\rho(T_D) \le \norm{(D-H)(D+H)^{-1}}.
\]
A natural idea is to seek $D$ so that the upper bound is small.
At first glance we may choose $D$ as the diagonal elements of $H$.
Inspired by the diagonal weighted matrix in Freund, 1992~\cite{Freund1992}, the Euclidean norms of column vectors could also be exploited.
However, the common experience is that these strategies may lead to a stagnation of convergence, and sometimes perform much worse than choosing~$\gamma=1$.
We will not pursue them further in this paper.

\subsection{Parameter estimation based on gradient iterations}
\label{sec:3.2}

It is observed that \eqref{eq:3.1} leads to a straightforward estimation of parameter~$\gamma_*$ in~\eqref{eq:gamma}.
From Figure~\ref{fig:1} we can deduce that the optimal parameter in HSS could be actually approximated by steepest descent iterations, which is shown in the following theorem.
\begin{theorem}
\label{thm:4}
Assume that the matrix~$H$ in system~\eqref{eq:ls:hpd} is the Hermitian part of~$A$ in system~\eqref{eq:ls}.
If steepest descent is used for solving~\eqref{eq:ls:hpd}, then the following limit holds
\begin{equation}
\label{eq:4.1}
\lim_{n\rightarrow\infty}\sqrt{\Gamma_n}=\gamma_*.
\end{equation}
\end{theorem}
\begin{proof}
Combining~\eqref{eq:gamma}, \eqref{eq:3.1} and the fact that~$\Gamma_n>0$ observed from~\eqref{eq:Q}, the desired conclusion follows.
\end{proof}

Another approach is to compute the approximation by combining Lemmas~\ref{thm:2} and~\ref{thm:3}, in which case $\gamma_*$ could be estimated without explicit access to operator~$H$.
This approach is shown in Theorem~\ref{thm:5}.
\begin{theorem}
\label{thm:5}
Assume that the matrix~$H$ in system~\eqref{eq:ls:hpd} is the Hermitian part of~$A$ in system~\eqref{eq:ls}.
If steepest descent is used for solving~$\mathcal{M}_1x=\hat{b}$, then the following limit holds
\begin{equation}
\label{eq:5.1}
\lim_{n\rightarrow\infty}\sqrt{\Gamma_n-\gamma\alpha_n^\text{RA}+\gamma}=\gamma_*.
\end{equation}
\end{theorem}
\begin{proof}
Recall that $\mathcal{M}_1=\alpha I + H$.
Since
\[
\lambda_i(\mathcal{M}_1) = \gamma + \lambda_i(H)
\]
for $i=1,\,\dots,\,N$, it follows that
\[
\begin{split}
\gamma_* &= \sqrt{\lambda_1(H)\lambda_N(H)} = \sqrt{(\lambda_1(\mathcal{M}_1)-\gamma)(\lambda_N(\mathcal{M}_1)-\gamma)} \\[1ex]
&= \sqrt{\lambda_1(\mathcal{M}_1)\lambda_N(\mathcal{M}_1) - \gamma(\lambda_1(\mathcal{M}_1)+\lambda_N(\mathcal{M}_1)) + \gamma^2}
\end{split}
\]
Combining~\eqref{eq:2.1} and~\eqref{eq:3.1} implies
\[
\begin{split}
\gamma_* &= \sqrt{\lim_{n\rightarrow\infty}\Gamma_n - \gamma\lim_{n\rightarrow\infty}\alpha_n^\text{RA} + \gamma^2} \\[1ex]
&= \lim_{n\rightarrow\infty}\sqrt{\Gamma_n - \gamma\alpha_n^\text{RA} + \gamma^2}.
\end{split}
\]
This completes out proof.
\end{proof}
\begin{remark}
Practically, obtaining~$\gamma_*$ by~\eqref{eq:5.1} requires a predetermined parameter~$\gamma$.
One could choose $\gamma=1$ and give an integer~$k$ as the maximum number of iterations such that
\[
\gamma_* \approx \sqrt{\Gamma_k - \alpha_k^\text{RA} + 1},
\]
in which case the HSS algorithm might be executed at reduced costs.
\end{remark}

Another direction of approach is based on the minimal gradient (MG) steplength
\[
\alpha_n^\text{MG} = \frac{g_n^\mathsf{H} Hg_n}{g_n^\mathsf{H} H^2g_n},
\]
the spectral properties of which have been discussed by the present authors along with several new gradient methods in a separate paper.
Let
\[
\alpha_n^\text{A2} = \left(\frac{1}{\alpha_{n-1}^\text{MG}}+\frac{1}{\alpha_n^\text{MG}}\right)^{-1},\quad \tilde{\Gamma}_n = \frac{1}{\alpha_{n-1}^\text{MG}\alpha_n^\text{MG}}-\frac{g_n^\mathsf{H}Hg_n}{\left(\alpha_{n-1}^\text{MG}\right)^2 g_{n-1}^\mathsf{H}Hg_{n-1}}.
\]
Let us write $\alpha_n^\text{RA2}=\left(\alpha_n^\text{A2}\right)^{-1}$.
\begin{theorem}
\label{thm:6}
Assume that the matrix~$H$ in system~\eqref{eq:ls:hpd} is the Hermitian part of~$A$ in system~\eqref{eq:ls}.
If minimal gradient is used for solving~\eqref{eq:ls:hpd}, then the following limit holds
\begin{equation}
\label{eq:6.1}
\lim_{n\rightarrow\infty}\sqrt{\tilde{\Gamma}_n}=\gamma_*.
\end{equation}
\end{theorem}
\begin{proof}
The proof can be obtained similarly as the one in Theorem~\ref{thm:4}.
\end{proof}
\begin{theorem}
\label{thm:7}
Assume that the matrix~$H$ in system~\eqref{eq:ls:hpd} is the Hermitian part of~$A$ in system~\eqref{eq:ls}.
If minimal gradient is used for solving~$\mathcal{M}_1x=\hat{b}$, then the following limit holds
\begin{equation}
\label{eq:7.1}
\lim_{n\rightarrow\infty}\sqrt{\tilde{\Gamma}_n-\gamma\alpha_n^\text{RA2}+\gamma}=\gamma_*.
\end{equation}
\end{theorem}
\begin{proof}
The proof can be obtained similarly as the one in Theorem~\ref{thm:5}.
\end{proof}

\subsection{Solution based on lagged gradient iterations}
\label{sec:3.3}

Although steepest descent has remarkable spectral properties, as an iterative method, its popularity has been overshadowed by CG.
Akaike~\cite{Akaike1959} exploited the fact that the zigzag behavior nearly always leads to slow convergence, except when initial gradient approaches an eigenvector.
This drawback can be cured with a lagged strategy, first proposed by Barzilai and Borwein~\cite{Barzilai1988}, which was later called Barzilai-Borwein (BB) method.
The idea is to provide a two-point approximation to the quasi-Newton methods, namely
\[
\alpha_n^\text{BB} = \argmin_{\alpha} \norm{\frac{1}{\alpha}\Delta x-\Delta g}^2,
\]
where $\Delta x=x_n-x_{n-1}$ and $\Delta g=g_n-g_{n-1}$, yielding
\[
\alpha_n^\text{BB} = \frac{g_{n-1}^\mathsf{H} g_{n-1}}{g_{n-1}^\mathsf{H} Hg_{n-1}}.
\]
Notice that $\alpha_n^\text{BB} = \alpha_{n-1}^\text{SD}$.
The convergence analysis was given in Raydan, 1993~\cite{Raydan1993} and Dai and Liao, 2002~\cite{Dai2002}.
For the $Q$-linear result, however, has never been proved due to its nonmonotone convergence.
It seems overall that the effect of this irregular behavior is beneficial.

For the HSS method, two iterative procedures are needed at each iteration.
Since the solution of subproblems in~\eqref{eq:hss} is sometimes as difficult as that of the original system~\eqref{eq:ls}, the inexact solvers with rather low precision are often considered, especially for ill-conditioned problems.
In practice, the first equation of~\eqref{eq:hss} is usually solve by CG, and the second equation of~\eqref{eq:hss} can be solved by CGNE~\cite{Saad2003}.
Friedlander et al.~\cite{Friedlander1999} made the observation that BB could often be competitive with CG when low precision is required.
It is known that CG is sensitive to rounding errors, while lagged gradient methods can remedy this issue~\cite{Fletcher2005,vandenDoel2012} with less computational costs per iteration.
Additionally, although BB sometimes suffers from the disadvantage of requiring increasing number of iterations for increasing condition numbers, its low-precision behavior tends to be less sensitive to the ill-conditioning.

A similar method developed by symmetry~\cite{Barzilai1988} is of the form
\[
\alpha_n^\text{BB2} = \frac{g_{n-1}^\mathsf{H} Hg_{n-1}}{g_{n-1}^\mathsf{H} H^2g_{n-1}},
\]
which imposes as well a quasi-Newton property
\[
\alpha_n^\text{BB2} = \argmin_{\alpha} \norm{\Delta x-\alpha\Delta g}^2.
\]
Notice that $\alpha_n^\text{BB2} = \alpha_{n-1}^\text{MG}$.
In the last three decades, much effort was devoted to develop new lagged gradient methods, see De~Asmundis et al., 2014~\cite{DeAsmundis2014} and the references therein.

An example is illustrated in Figure~\ref{fig:2}.
\begin{figure}[t]
\centerline{\includegraphics[width=.7\linewidth]{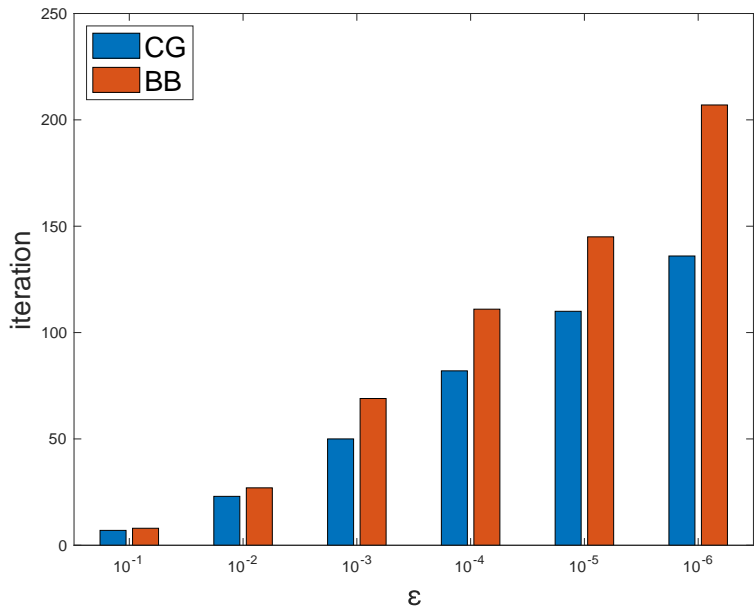}}
\caption{Comparison of CG and BB for solving system~\eqref{eq:ls:hpd} where $H$ is a diagonal matrix of size $10^3$ with $\kappa(H)=10^3$ and $\hat{b}$ is a vector of all ones.\label{fig:2}}
\end{figure}
We solve~\eqref{eq:ls:hpd} with different residual thresholds~$\varepsilon$, where $H$ is chosen as a diagonal matrix of size~$10^3$ and $\hat{b}$ is a vector of all ones.
The diagonal entries have values logarithmically distributed between $10^{-3}$ and $1$ in ascending order, with the first and the last entries equal to the limits, respectively, such that~$\kappa(H)=10^3$.
The plot shows a fairly efficient behavior of BB.

\section{Numerical experiments}
\label{sec:4}

In this section we perform some numerical tests.
Assume that iterative algorithms are started from zero vectors.
The global stopping criterion in HSS is determined by the threshold $\varepsilon = \norm{b-Ax_n}/\norm{b}$ with a fixed convergence tolerance~$10^{-6}$.
The inner stopping thresholds $\varepsilon_1$ and $\varepsilon_2$ for the two half-steps of~\eqref{eq:hss} are defined in the same way.
For gradient iterations applied to system~\eqref{eq:ls:hpd}, similarly, the stopping criterion is defined by the threshold~$\varepsilon = \lVert\hat{b}-Hx_n\lVert/\lVert\hat{b}\lVert$ with the same tolerance.
All tests are run in double precision.

\subsection{Asymptotic results of gradient iterations}
\label{sec:4.1}

The goal of the first experiment is to illustrate how the spectral properties described earlier can be used for providing a rough estimate of parameter~$\gamma_*$.
We have implemented steepest descent and minimal gradient iterations for several real matrices of size~$1000$ generated by MATLAB routine~\texttt{sprandsym}.
The right-hand side is chosen to be a vector of ones.
In Figure~\ref{fig:3}, parameter~$\gamma$ is plotted versus iteration number, under which a red dotted line marks out the position of~$\gamma_*$.
\begin{figure}[t]
\centering
\begin{subfigure}{.5\textwidth}
  \centering
  \includegraphics[width=.8\linewidth]{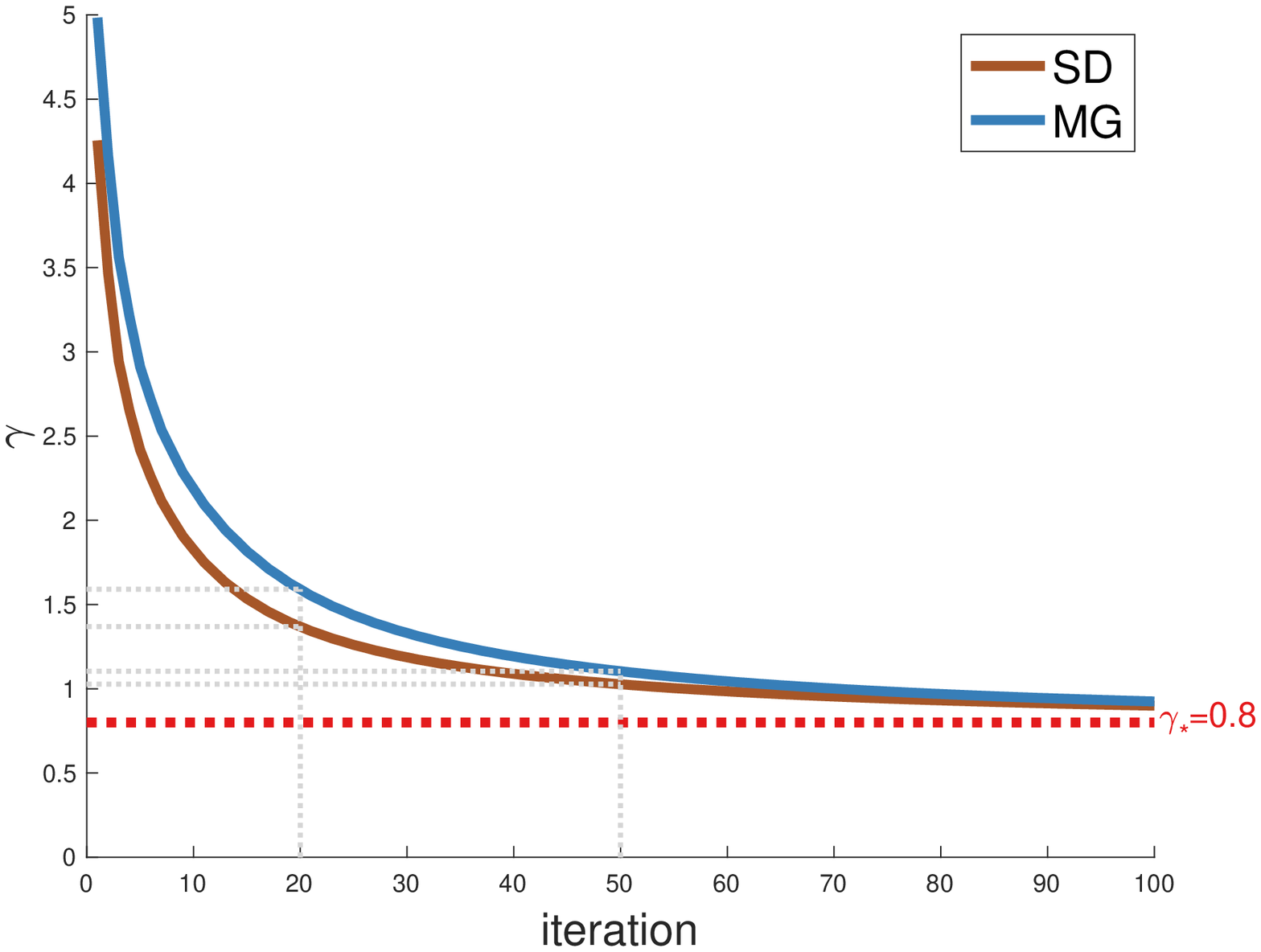}
\end{subfigure}\begin{subfigure}{.5\textwidth}
  \centering
  \includegraphics[width=.8\linewidth]{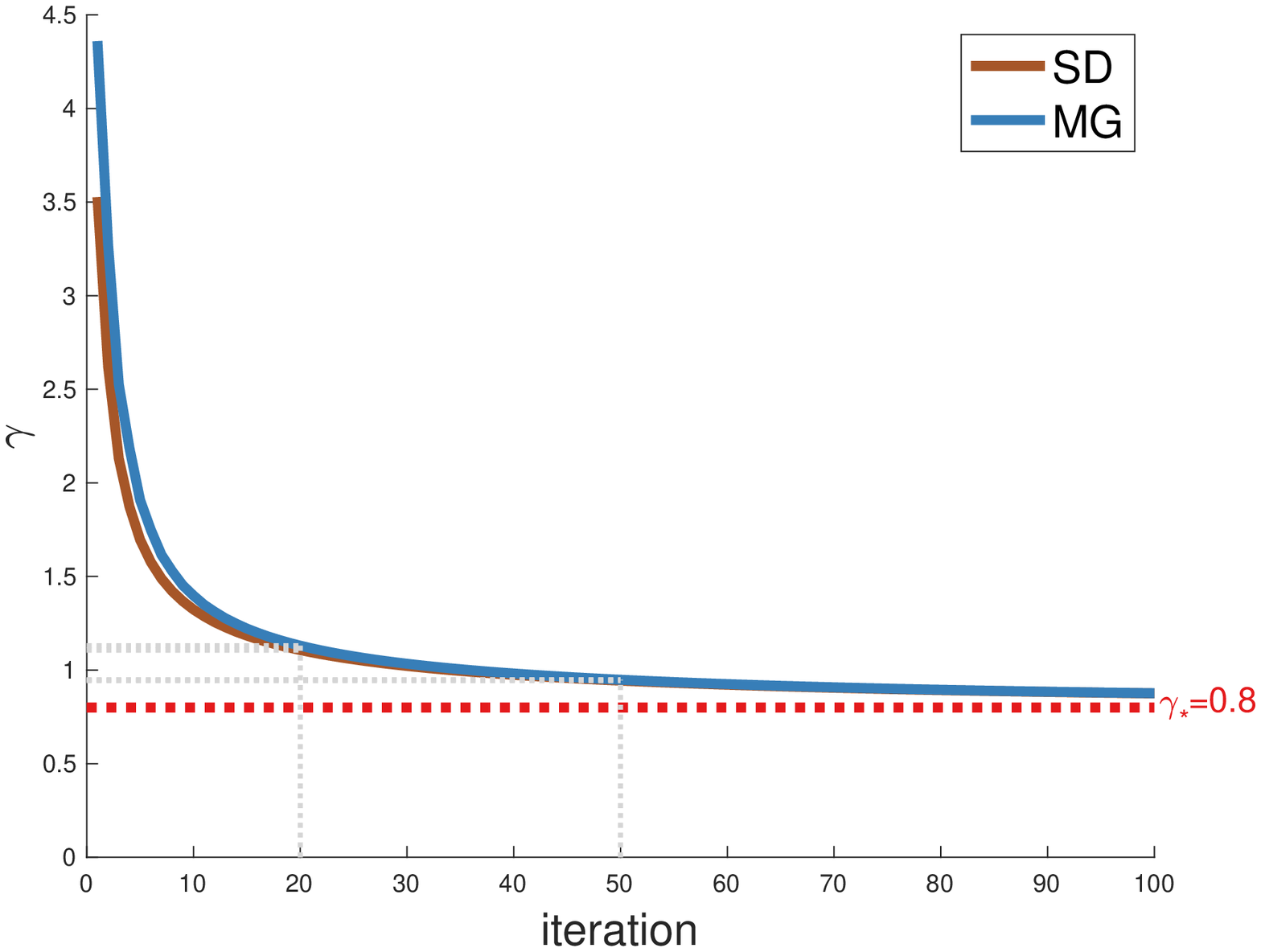}
\end{subfigure}
\begin{subfigure}{.5\textwidth}
  \centering
  \includegraphics[width=.8\linewidth]{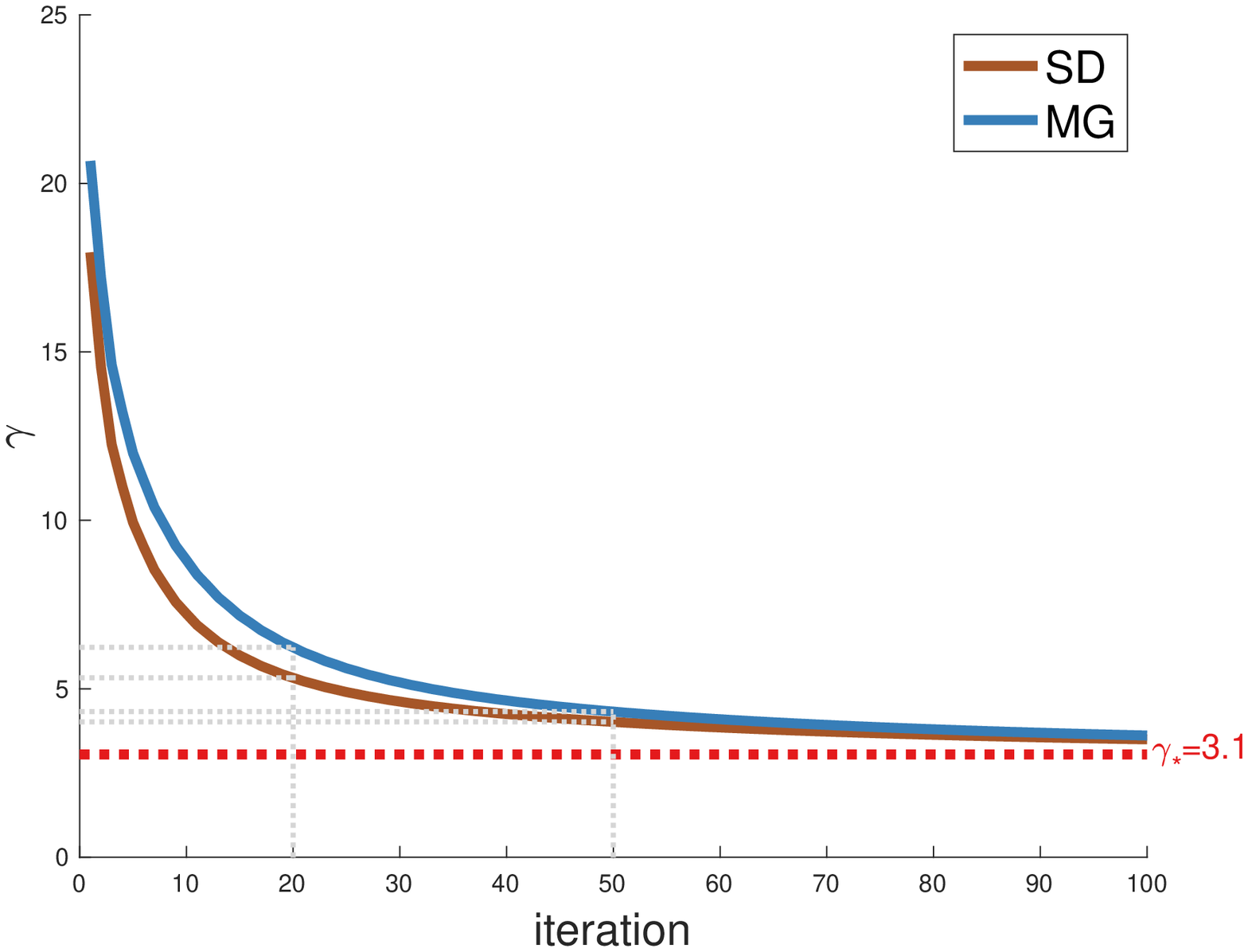}
\end{subfigure}\begin{subfigure}{.5\textwidth}
  \centering
  \includegraphics[width=.8\linewidth]{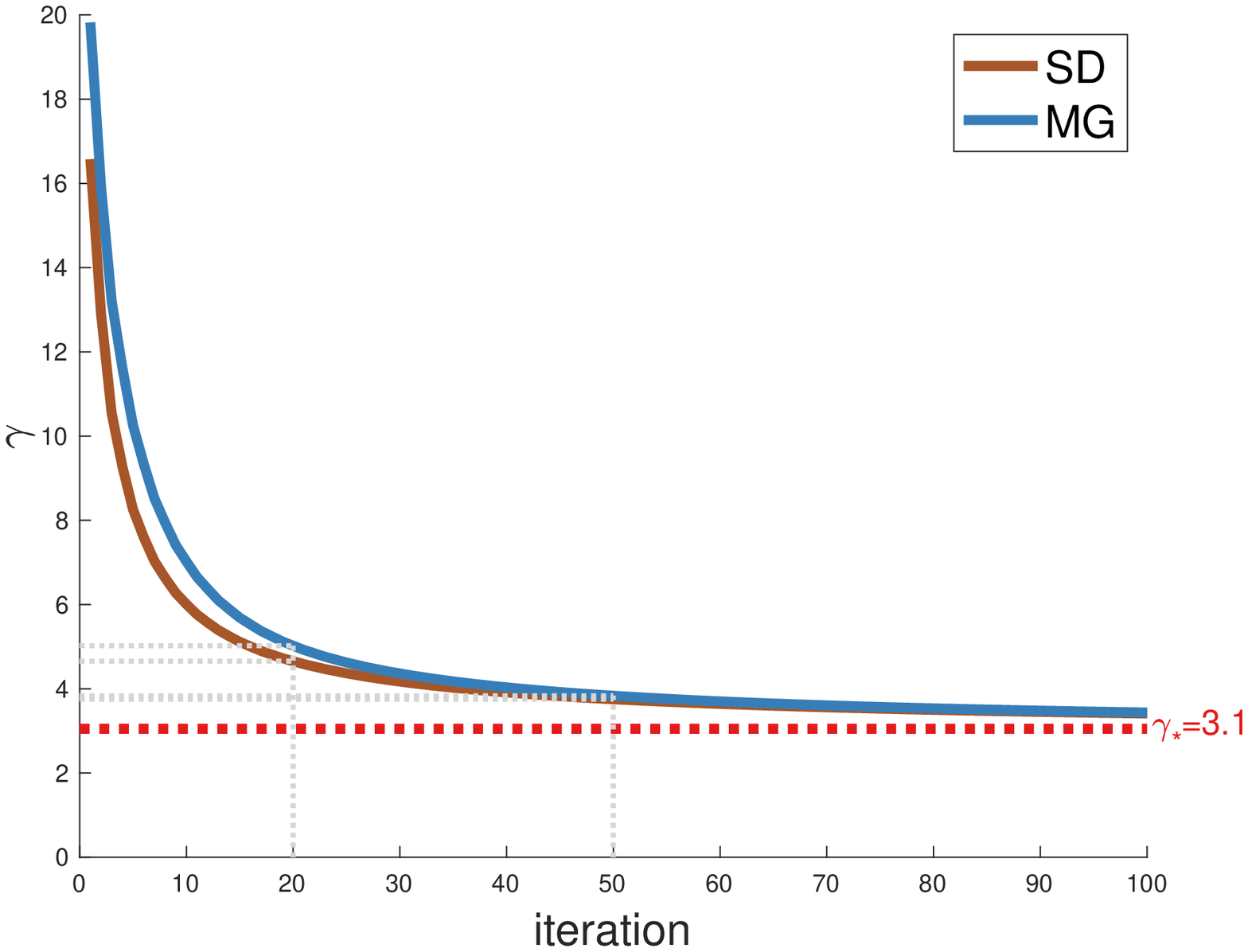}
\end{subfigure}
\begin{subfigure}{.5\textwidth}
  \centering
  \includegraphics[width=.8\linewidth]{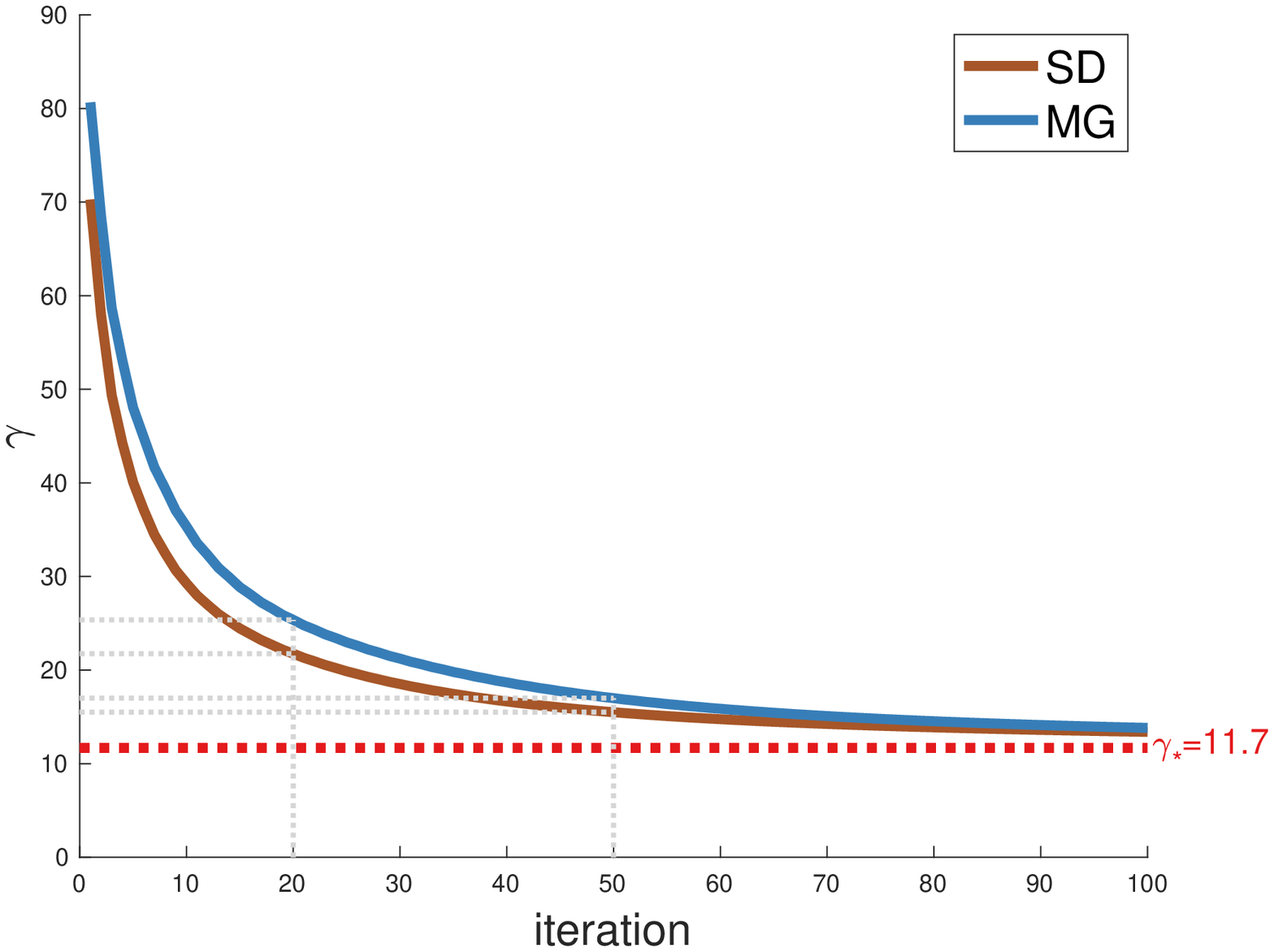}
\end{subfigure}\begin{subfigure}{.5\textwidth}
  \centering
  \includegraphics[width=.8\linewidth]{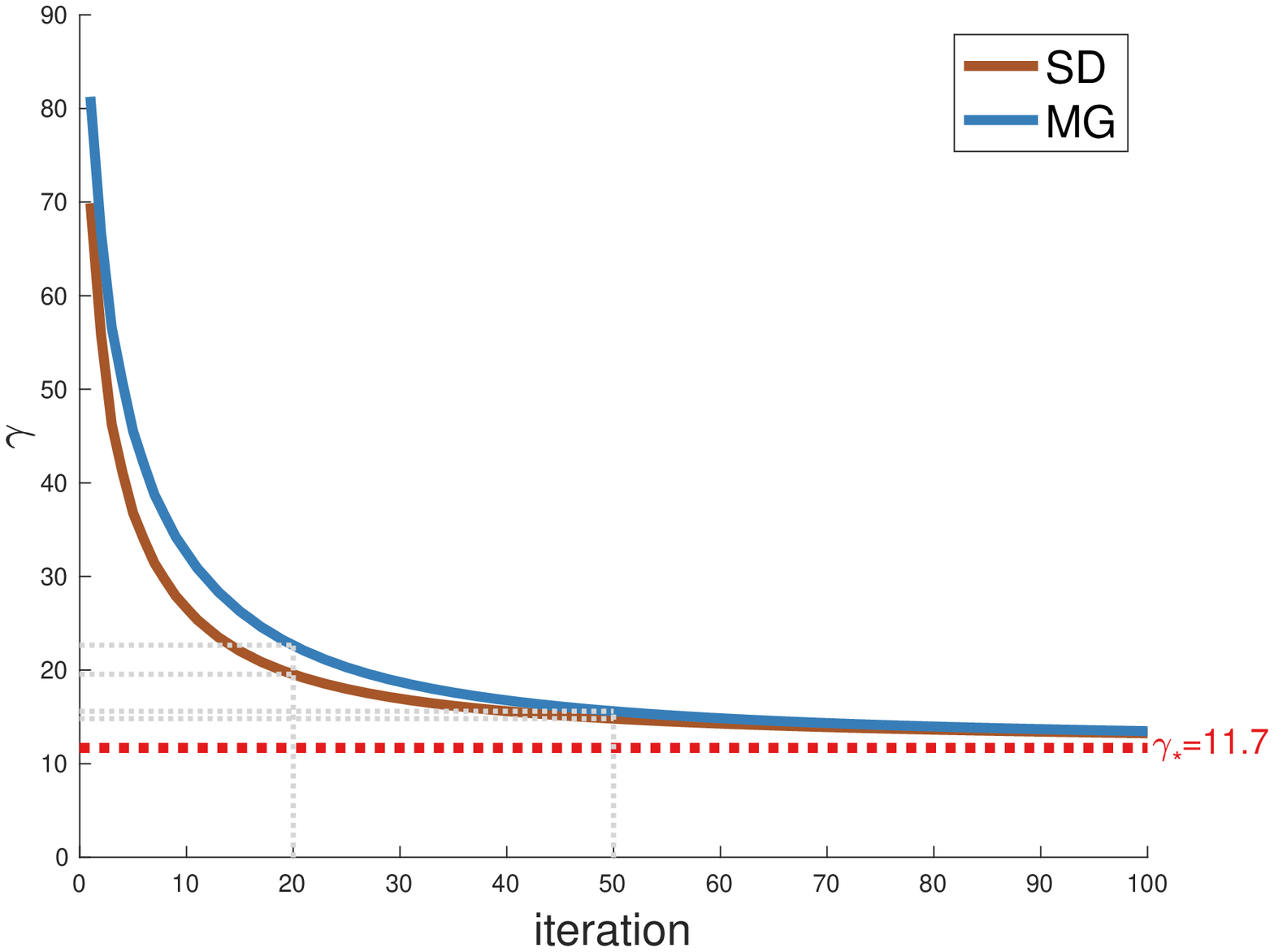}
\end{subfigure}
\caption{Parameter estimation with different matrix~$H$ generated randomly by MATLAB: $\gamma_*=0.8$ (top), $\gamma_*=3.1$ (middle), $\gamma_*=11.7$ (bottom). Parameter~$\gamma$ is computed by two approaches: Theorems~\ref{thm:4} and~\ref{thm:6} (left), Theorems~\ref{thm:5} and~\ref{thm:7} (right).}
\label{fig:3}
\end{figure}
It is clear that $\gamma$ tends to $\gamma_*$ asymptotically as expected.
As can be seen, steepest descent with limits~\eqref{eq:4.1} and~\eqref{eq:5.1} turns out to be a better strategy than minimal gradient in all cases.
The indirect approximations based on~\eqref{eq:5.1} and~\eqref{eq:7.1} yield faster convergence for both steepest descent and minimal gradient.

This test confirms Theorems~\ref{thm:4} to~\ref{thm:7}.
Recall that choosing $\gamma_*$ as parameter leads to an upper bound of $\rho(T)$, for which it is not necessary to obtain an exact estimate.
Experience shows that this choice may sometimes cause overfitting, resulting in slow convergence or even divergence, especially when $\gamma_*$ is small.
One simple measure is to use early stopping in gradient iterations.
In the following, it is assumed that steepest descent is used for parameter estimation in HSS, called preadaptive iterations, and we consider only the direct approach~\eqref{eq:4.1}.

\subsection{HSS with different parameters}
\label{sec:4.2}

In this test we generate some matrices obtained from a classical problem in order to understand the convergence behavior of HSS enhanced by steepest descent iterations.
\begin{example}
\label{ex:cd3d}
Consider system~\eqref{eq:ls} where $A$ arises from the discretization of partial differential equation
\begin{equation}
\label{eq:cd3d}
-\left(\frac{\partial^2u}{\partial x^2}+\frac{\partial^2u}{\partial y^2}+\frac{\partial^2u}{\partial z^2}\right) + \theta\left(\frac{\partial u}{\partial x}+\frac{\partial u}{\partial y}+\frac{\partial u}{\partial z}\right) = q
\end{equation}
on the unit cube $\Omega = [0,\,1]^3$ with $\theta$ a positive constant.
Assume that $u$ satisfies homogeneous Dirichlet boundary conditions.
The finite difference discretization on a uniform $m\times m\times m$ grid with mesh size $h = 1/(m+1)$ is applied to the above model yielding a linear system with $N=m^3$.
\end{example}

In the following we use the centered difference scheme for discretization.
The right-hand side~$b$ is generated with random complex values ranging in~$[-10,\,10]+\iota[-10,\,10]$.
As thresholds for inner iterations, $\varepsilon_1=10^{-4}$ and $\varepsilon_2=10^{-4}$ are chosen.
CG is exploited for solving the Hermitian inner system, while CGNE is used for the skew-Hermitian part.
Figure~\ref{fig:4} shows the convergence behavior of HSS upon different values of the parameter.
\begin{figure}[t]
\centerline{\includegraphics[width=.7\linewidth]{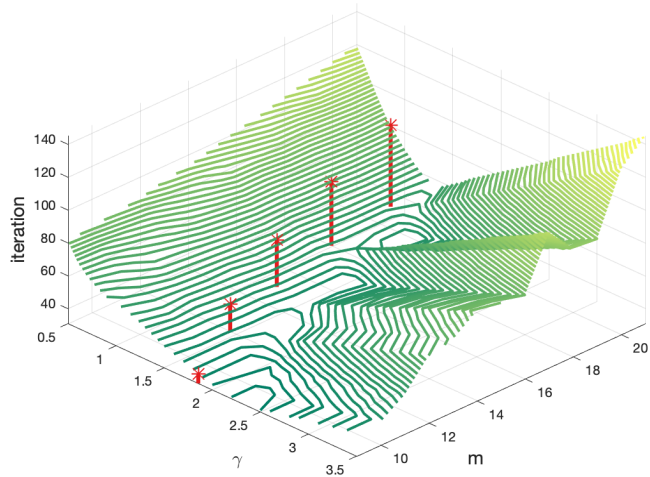}}
\caption{Solving problem~\eqref{eq:cd3d} by HSS with~$\gamma\in[0.5,\,3.5]$ and~$m\in[9,\,21]$. The optimal parameters $\gamma_*$ are located by red lines.\label{fig:4}}
\end{figure}
Here, we set~$\gamma\in[0.5,\,3.5]$ and~$m\in[9,\,21]$.
The optimal parameters~$\gamma_*$ with~$m=9,\,12,\,15,\,18,\,21$ are located by red lines.
Notice that a path that zigzags through the bottom of the valley corresponds to the best parameters.
As already noted that the parameter estimates need not be accurate, and thus the red lines are good enough in practice.

Then approximating $\gamma_*$ by inexact steepest descent iterations yields the preadaptive HSS method (PAHSS).
Let $\eta$ denote the number of preadaptive iterations.
The convergence behaviors and total computing times are illustrated in Figure~\ref{fig:5}.
\begin{figure}
\centering
\begin{subfigure}{.5\textwidth}
  \centering
  \includegraphics[width=.8\linewidth]{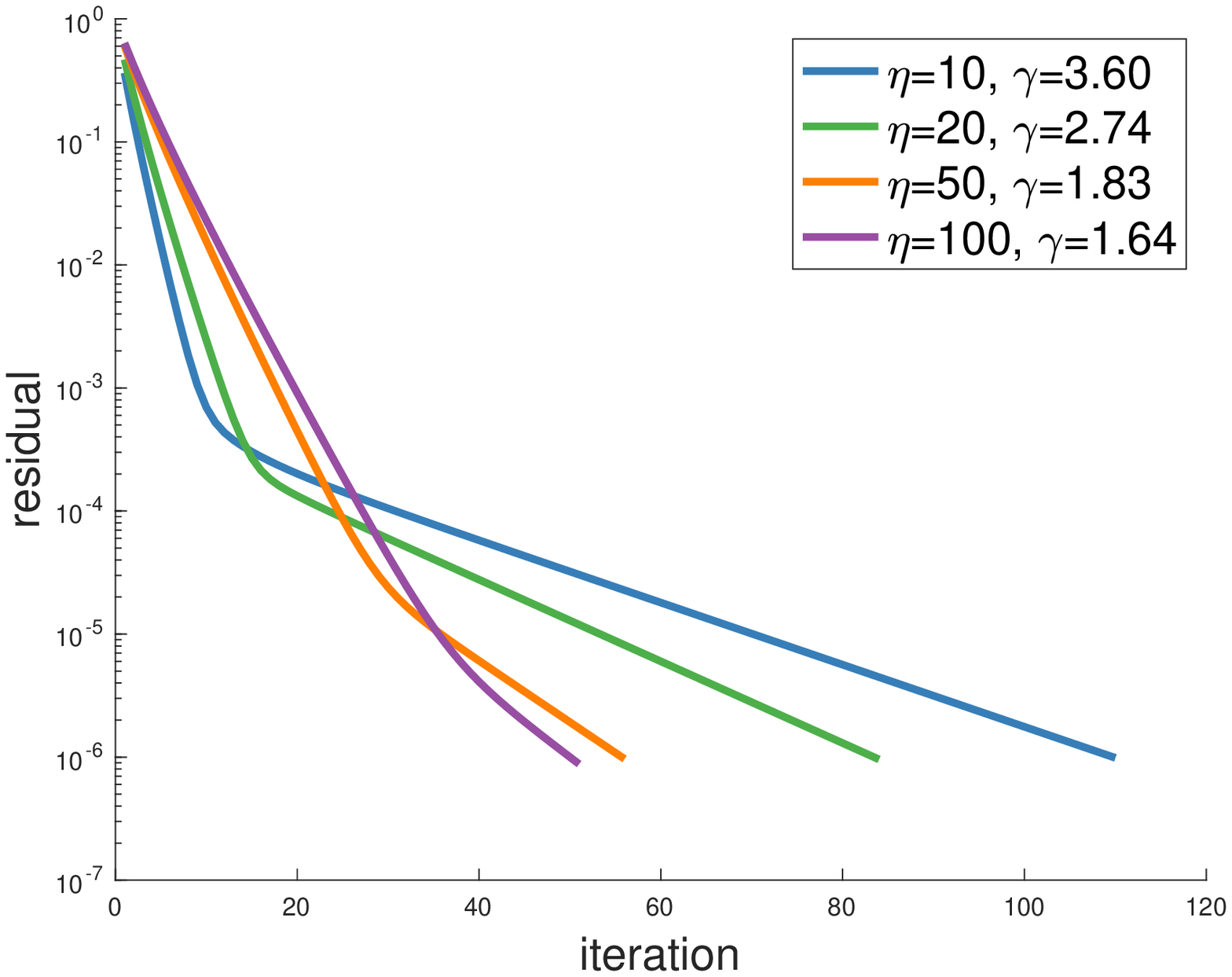}
\end{subfigure}\begin{subfigure}{.5\textwidth}
  \centering
  \includegraphics[width=.8\linewidth]{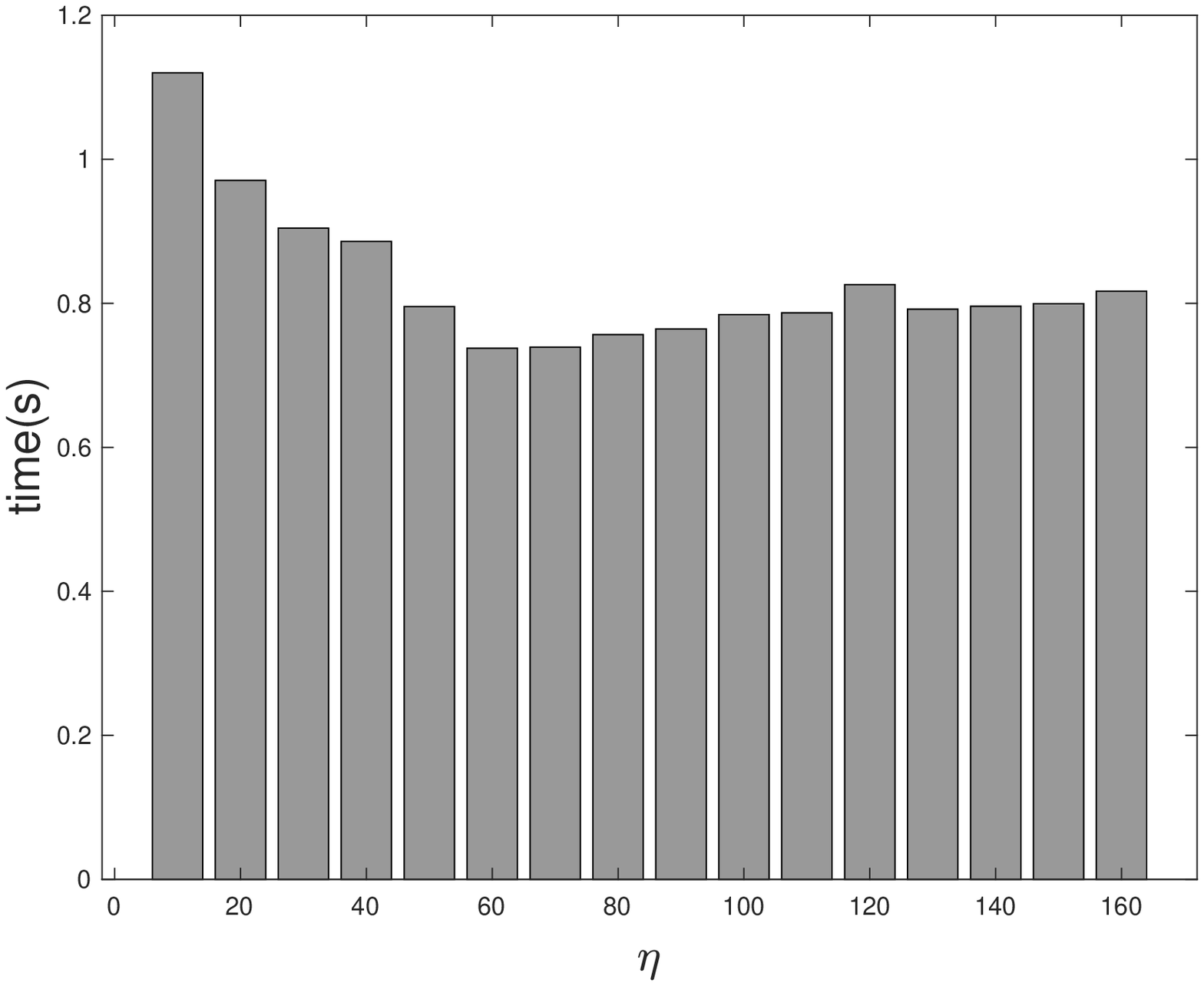}
\end{subfigure}
\begin{subfigure}{.5\textwidth}
  \centering
  \includegraphics[width=.8\linewidth]{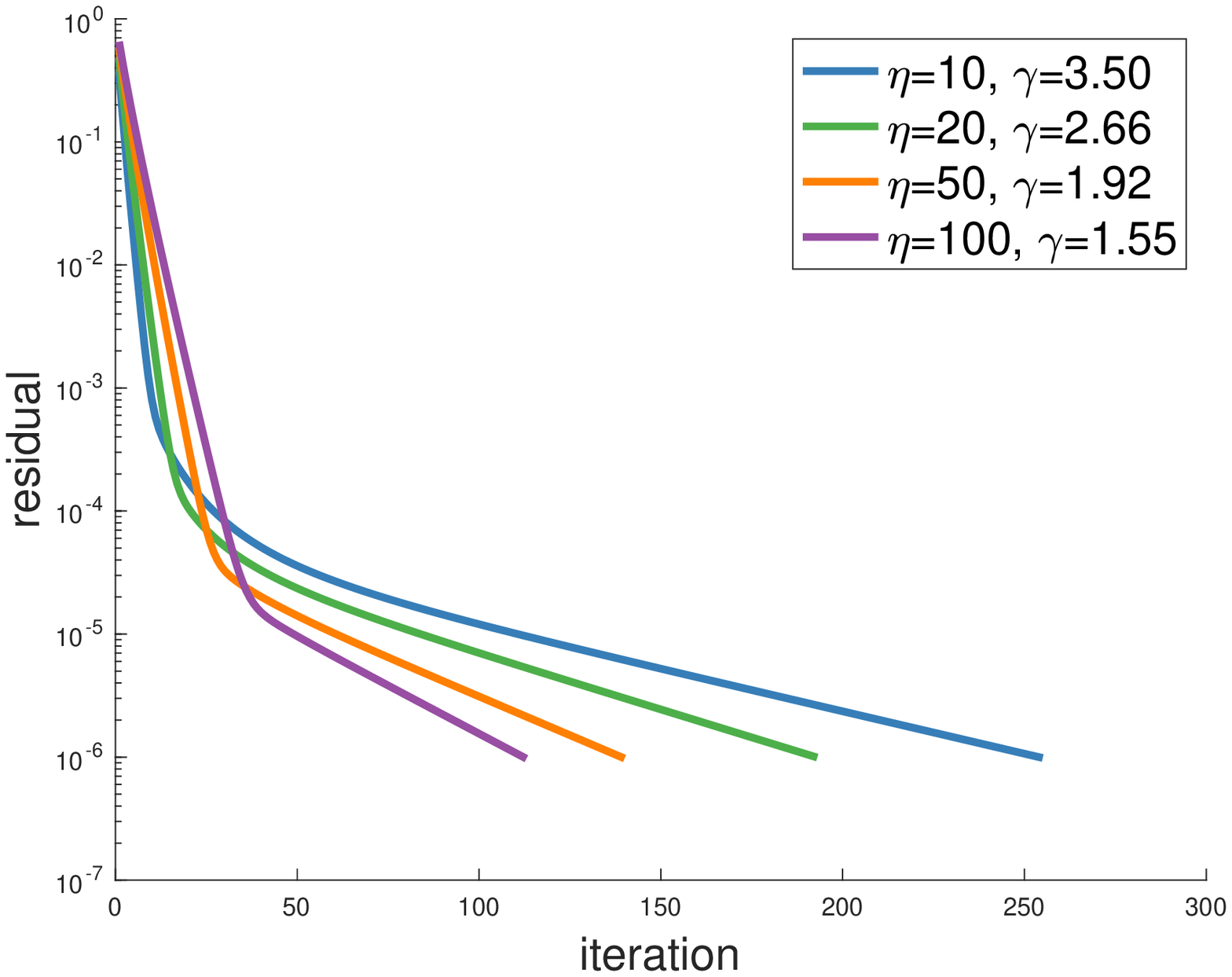}
\end{subfigure}\begin{subfigure}{.5\textwidth}
  \centering
  \includegraphics[width=.8\linewidth]{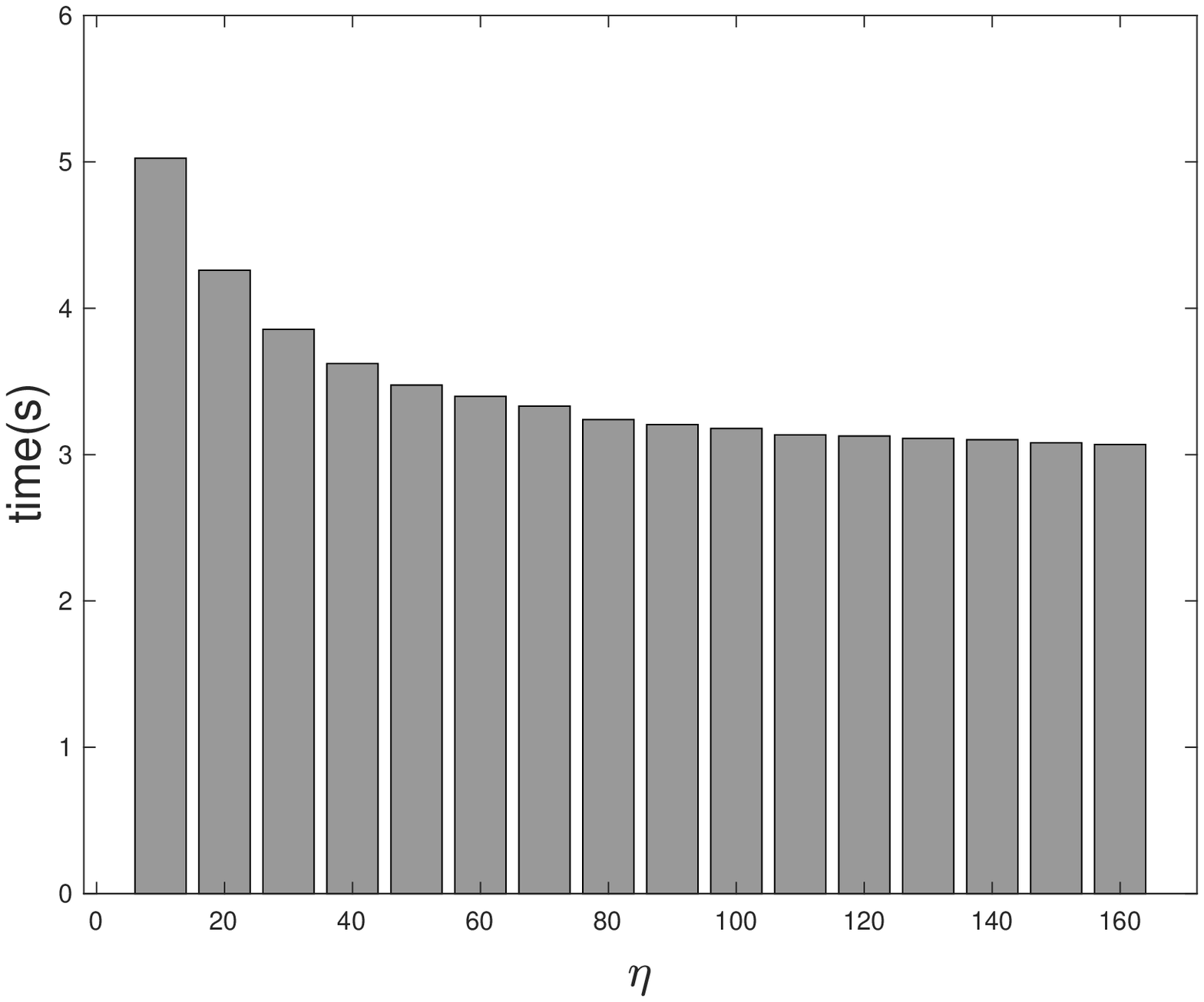}
\end{subfigure}
\begin{subfigure}{.5\textwidth}
  \centering
  \includegraphics[width=.8\linewidth]{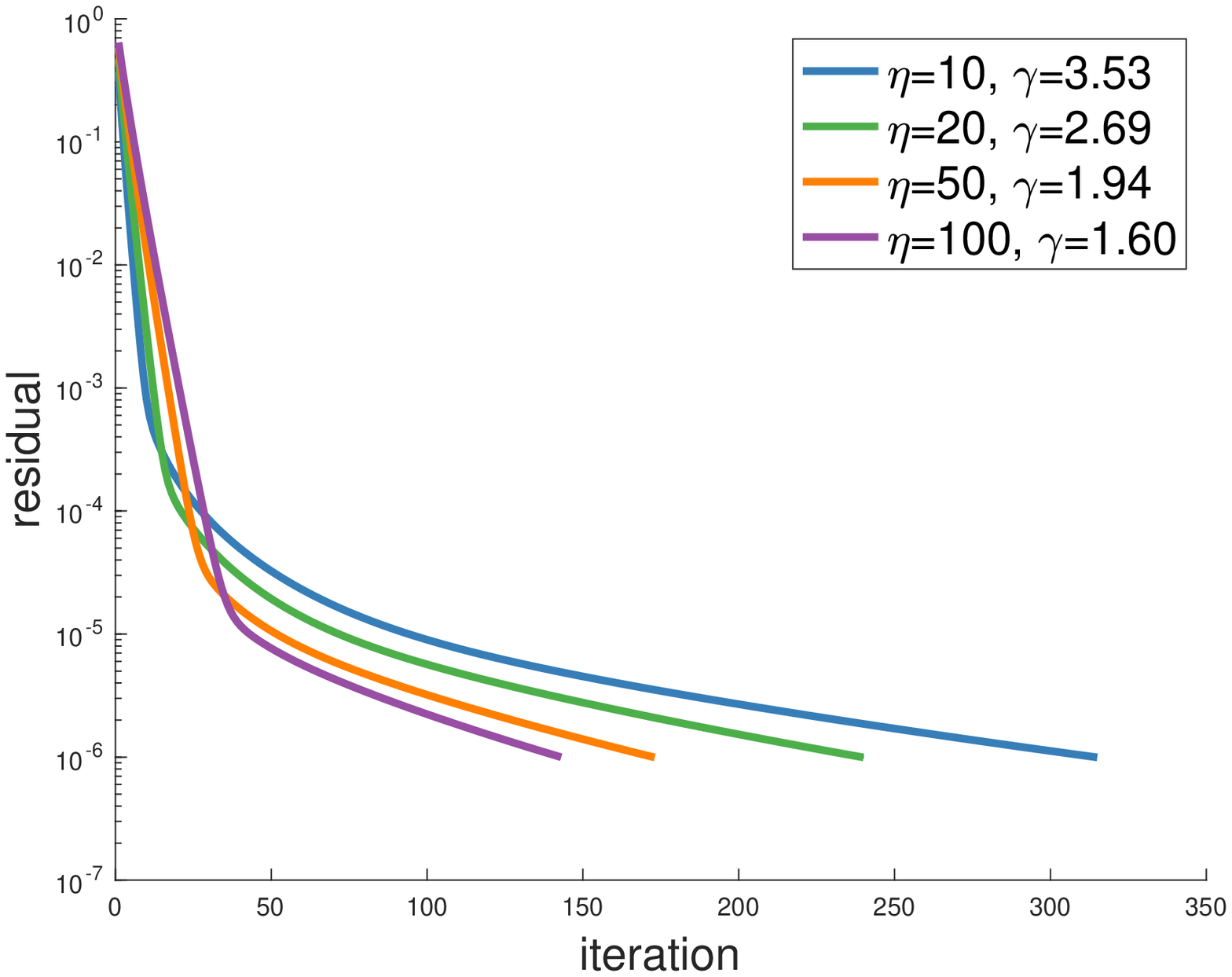}
\end{subfigure}\begin{subfigure}{.5\textwidth}
  \centering
  \includegraphics[width=.8\linewidth]{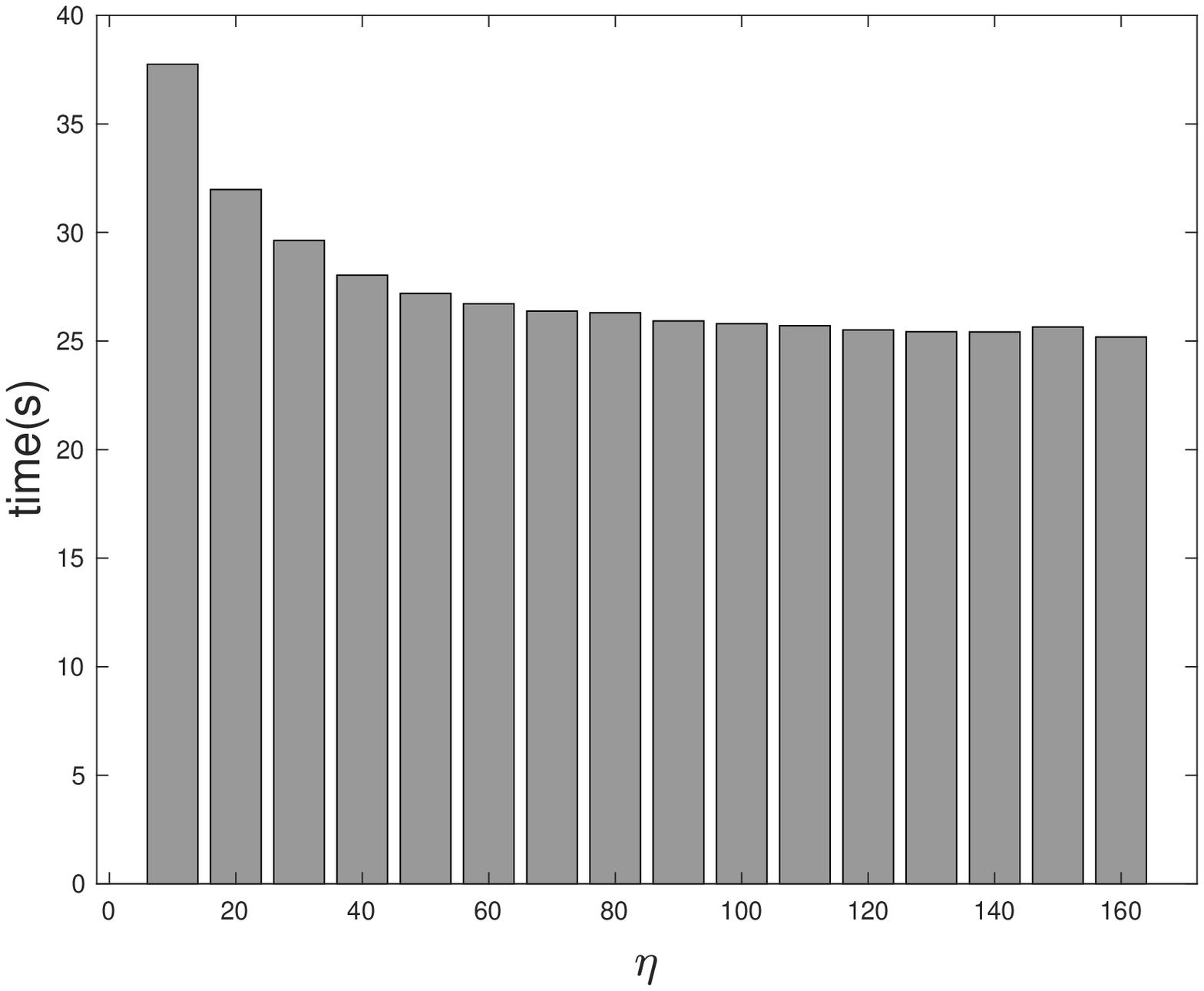}
\end{subfigure}
\begin{subfigure}{.5\textwidth}
  \centering
  \includegraphics[width=.8\linewidth]{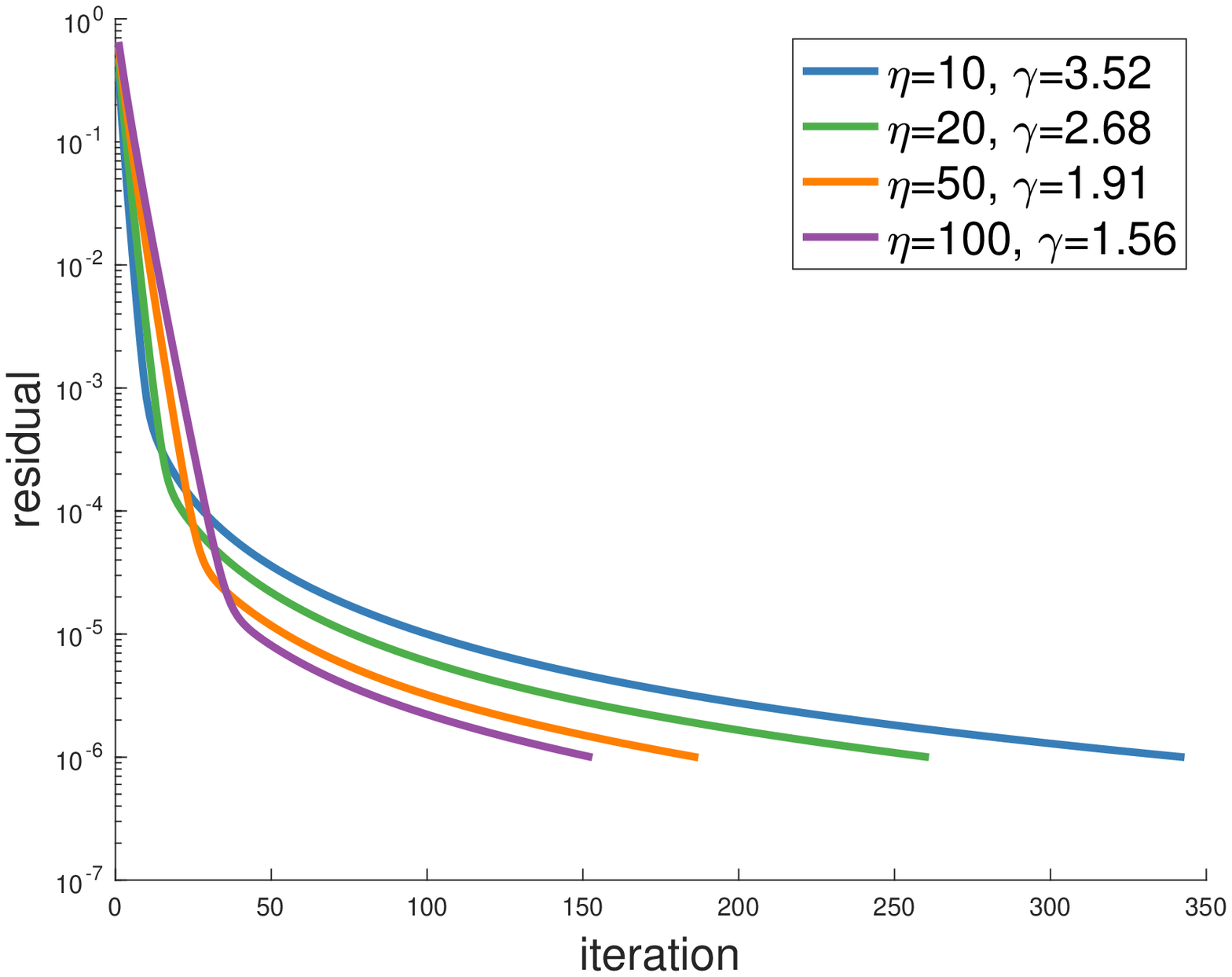}
\end{subfigure}\begin{subfigure}{.5\textwidth}
  \centering
  \includegraphics[width=.8\linewidth]{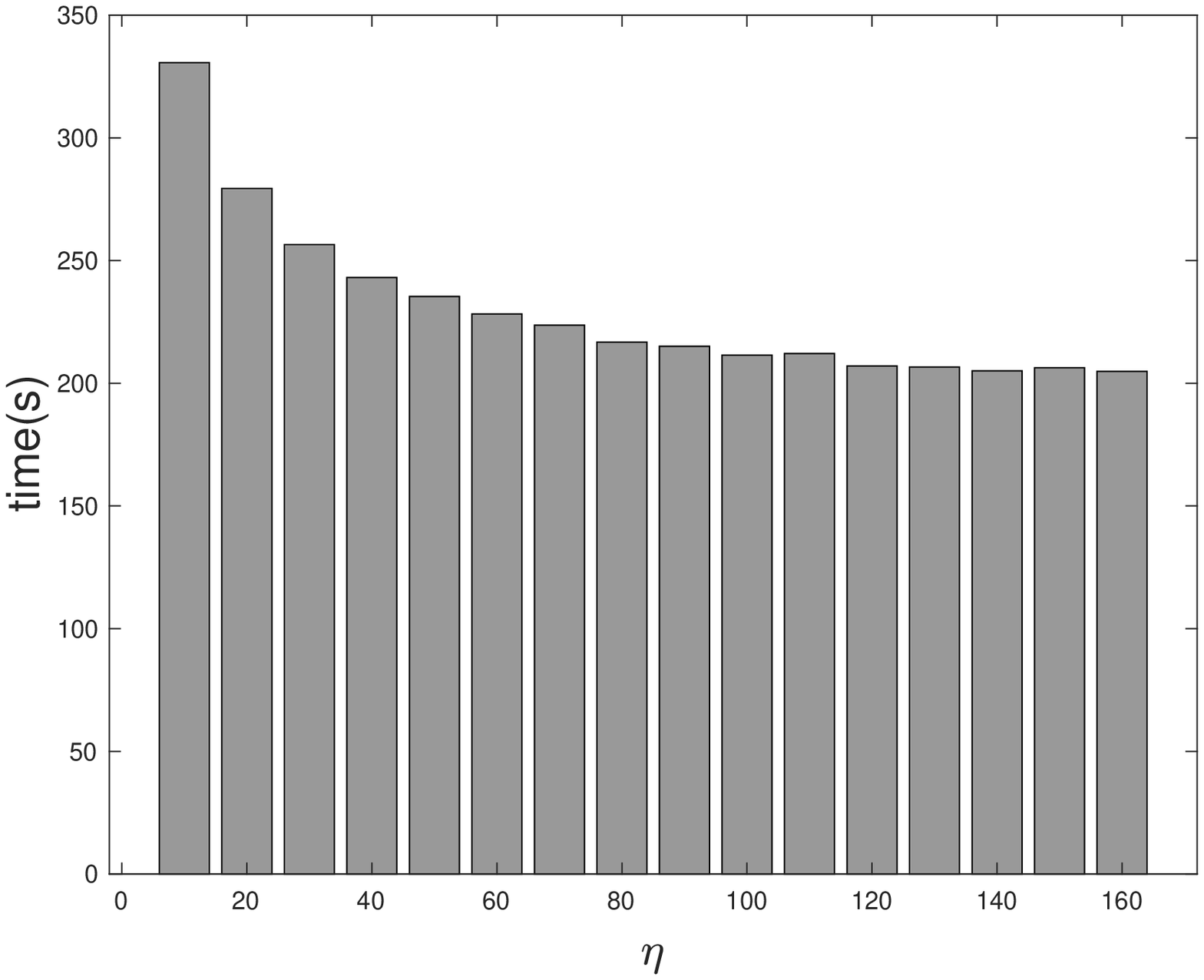}
\end{subfigure}
\caption{Parameter estimation for problem~\eqref{eq:cd3d} with different mesh densities: $m=16$ (first), $m=32$ (second), $m=64$ (third), $m=128$ (fourth). Left: convergence curves for different~$\eta$. Right: average wall-clock times for different~$\eta$ including that of steepest descent iterations.}
\label{fig:5}
\end{figure}
The left four plots show the residual curves with several typical choices of~$\eta$ when $m=16,\,32,\,64,\,128$, namely, $N=4096,\,32768,\,262144,\,2097152$.
Two observations can be made for all dimensions: the first is that larger $\eta$ yields faster convergence of HSS; the second is that $\eta=100$ does not lead to significant gains in efficiency compared with~$\eta=50$.
The right four plots show total wall-clock times of PAHSS iterations, measured in seconds, upon $\eta$ ranging from~$10$ to~$160$.
It can be seen that substantial gains are made in the beginning, following a long period of stagnation.
Experience shows that a small number of steepest descent iterations is sufficient and it is therefore appropriate to use early stopping.

\subsection{CG and BB as low-precision inner solvers}
\label{sec:4.3}

In order to verify that BB can be an efficient alternative to CG as low-precision inner solver for HSS, some tests proceed along the same lines as above but consider both CG and BB as inner solvers for the Hermitian part.
Numbers of total iterations and wall-clock times, which are measured in seconds, are shown in Table~\ref{tab:1}.
\begin{table*}[t]
\caption{Results of different methods for problem~\eqref{eq:cd3d} with $\varepsilon_1=10^{-1}$, $\varepsilon_2=10^{-4}$ and $\gamma=1$.\label{tab:1}}
\centering
\begin{tabular*}{400pt}{@{\extracolsep\fill}l | lll | lll | lll@{\extracolsep\fill}}
\toprule
& \multicolumn{3}{@{}c|@{}}{HSS-CG} & \multicolumn{3}{@{}c|@{}}{HSS-BB} & \multicolumn{3}{@{}c@{}}{ORTHODIR} \\
\midrule
$m$ & $40$ & $60$ & $80$ & $40$ & $60$ & $80$ & $40$ & $60$ & $80$ \\
\midrule
\# iters & $446$ & $475$ & $535$ & $609$ & $686$ & $769$ & $80$ & $85$ & $84$ \\
time (s) & $5.782$ & $6.379$ & $6.832$ & $6.054$ & $7.048$ & $7.659$ & $7.760$ & $8.385$ & $8.350$ \\
\bottomrule
\end{tabular*}
\end{table*}
Since the optimal parameters $\gamma_*$ for~\eqref{eq:cd3d} with $m=40,\,60,\,80$ are less that $1$, which may lead to stability problem, we choose $\gamma=1$ for all tests.
We conduct $10$ repeated experiments and print only the average computation times.
We also add here the results of ORTHODIR~\cite{Young1980} for solving~\eqref{eq:ls} for the purpose of comparison.
The comparison of costs is shown in Table~\ref{tab:2}.
\begin{table*}[t]
\caption{Summary of operations for iteration $i$ and storage requirements. In the HSS row, the term ``solver'' represents an inner solver like CG, BB or CGNE, while the storage requires Hermitian and skew-Hermitian parts of $A$ and the residual vector.\label{tab:2}}
\centering
\begin{tabular*}{400pt}{@{\extracolsep\fill}lllll@{\extracolsep\fill}}
\toprule
method & dot products & vector updates & matrix-vector & storage \\
\midrule
CG & $2$ & $3$ & $1$ & $4N$ \\
BB & $2$ & $2$ & $1$ & $3N$ \\
CGNE & $2$ & $3$ & $2$ & $3N$ \\
HSS & solver & solver$+2$ & solver$+1$ & $2$ matrices $+N$ \\
ORTHODIR & $i+2$ & $2i+2$ & $1$ & $2i+5N$ \\
\bottomrule
\end{tabular*}
\end{table*}
As expected, BB is less efficient than CG in terms of computation times but BB shows a clear advantage for storage requirements and resistance to perturbation~\cite{Fletcher2005,vandenDoel2012}.
In addition, BB and CG used within HSS make the HSS method better than ORTHODIR.
The major drawback to ORTHODIR is that the computational work and storage requirement per iteration rise linearly with the iteration number.
This drawback can be reduced with a restarted version of the ORTHODIR, but this increases the total number of iterations and in the end does not reduce the computation time.
The choice between HSS and Krylov subspace methods depends on how expensive the sparse matrix-vector multiplications are in comparison to the vector updates and how much storage is available for the routine.

\section{Conclusion}
\label{sec:5}

Gradient iterations provide a versatile tool in linear algebra.
Apart from parameter estimates related to the spectral properties, steepest descent variants have also been tried recently with success as iterative methods~\cite{DeAsmundis2013,DeAsmundis2014,Gonzaga2016}.
This paper extends the spectral properties of gradient iterations and gives an application in the Hermitian and skew-Hermitian splitting method.
Note that this approach can be extended to other splitting methods (see Section~\ref{sec:1} and the references therein) without difficulty where a parameter $\gamma$ is needed to be computed.
Our experiments confirm that the gradient-enhanced HSS method can be an attractive alternative to the original one.

\section*{Acknowledgments}

This work was supported by the French national programme LEFE/INSU and the project ADOM (M\'ethodes
de d\'ecomposition de domaine asynchrones) of the French National Research Agency (ANR).

\bibliography{ref}

\begin{thebibliography}{10}

\bibitem{Akaike1959}
H.~Akaike.
\newblock On a successive transformation of probability distribution and its
  application to the analysis of the optimum gradient method.
\newblock {\em Ann. Inst. Stat. Math.}, 11(1):1--16, 1959.

\bibitem{Bai2008}
Z.-Z. Bai.
\newblock Several splittings for non-{H}ermitian linear systems.
\newblock {\em Sci. China Ser. A}, 51(8):1339--1348, 2008.

\bibitem{Bai2003}
Z.-Z. Bai, G.~H. Golub, and M.~K. Ng.
\newblock Hermitian and skew-{H}ermitian splitting methods for non-{H}ermitian
  positive definite linear systems.
\newblock {\em SIAM J. Matrix Anal. Appl.}, 24(3):603--626, 2003.

\bibitem{Bai2007b}
Z.-Z. Bai, G.~H. Golub, and M.~K. Ng.
\newblock On successive-overrelaxation acceleration of the {H}ermitian and
  skew-{H}ermitian splitting iterations.
\newblock {\em Numer. Linear Algebra Appl.}, 14(4):319--335, 2007.

\bibitem{Barzilai1988}
J.~Barzilai and J.~M. Borwein.
\newblock Two-point step size gradient methods.
\newblock {\em IMA J. Numer. Anal.}, 8(1):141--148, 1988.

\bibitem{Benzi2009}
M.~Benzi.
\newblock A generalization of the {H}ermitian and skew-{H}ermitian splitting
  iteration.
\newblock {\em SIAM J. Matrix Anal. Appl.}, 31(2):360--374, 2009.

\bibitem{Bertaccini2005}
D.~Bertaccini, G.~H. Golub, S.~S. Capizzano, and C.~T. Possio.
\newblock Preconditioned {HSS} methods for the solution of non-{H}ermitian
  positive definite linear systems and applications to the discrete
  convection-diffusion equation.
\newblock {\em Numer. Math.}, 99(3):441--484, 2005.

\bibitem{Cauchy1847}
A.~L. Cauchy.
\newblock M{\'e}thode g{\'e}n{\'e}rale pour la r{\'e}solution des syst{\`e}mes
  d'{\'e}quations simultan{\'e}es.
\newblock {\em Comp. Rend. Sci. Paris}, 25(1):536--538, 1847.
\newblock (in French).

\bibitem{Dai2002}
Y.-H. Dai and L.-Z. Liao.
\newblock {$R$}-linear convergence of the {B}arzilai and {B}orwein gradient
  method.
\newblock {\em IMA J. Numer. Anal.}, 22(1):1--10, 2002.

\bibitem{Dai2005c}
Y.-H. Dai and Y.-X. Yuan.
\newblock Analysis of monotone gradient methods.
\newblock {\em J. Ind. Manag. Optim.}, 1(2):181--192, 2005.

\bibitem{DeAsmundis2014}
R.~{De Asmundis}, D.~{di Serafino}, W.~W. Hager, G.~Toraldo, and H.~Zhang.
\newblock An efficient gradient method using the {Y}uan steplength.
\newblock {\em Comput. Optim. Appl.}, 59(3):541--563, 2014.

\bibitem{DeAsmundis2013}
R.~{De Asmundis}, D.~{di Serafino}, F.~Riccio, and G.~Toraldo.
\newblock On spectral properties of steepest descent methods.
\newblock {\em IMA J. Numer. Anal.}, 33(4):1416--1435, 2013.

\bibitem{Fletcher2005}
R.~Fletcher.
\newblock On the {B}arzilai-{B}orwein method.
\newblock In L.~Qi, K.~Teo, and X.~Yang, editors, {\em Optimization and Control
  with Applications}, pages 235--256. Springer, Boston, MA, 2005.

\bibitem{Freund1992}
R.~W. Freund.
\newblock Conjugate gradient-type methods for linear systems with complex
  symmetric coefficient matrices.
\newblock {\em SIAM J. Sci. Stat. Comput.}, 13(1):425--448, 1992.

\bibitem{Friedlander1999}
A.~Friedlander, J.~M. Mart{\'i}nez, B.~Molina, and M.~Raydan.
\newblock Gradient method with retards and generalizations.
\newblock {\em SIAM J. Numer. Anal.}, 36(1):275--289, 1999.

\bibitem{Gonzaga2016}
C.~C. Gonzaga and R.~M. Schneider.
\newblock On the steepest descent algorithm for quadratic functions.
\newblock {\em Comput. Optim. Appl.}, 63(2):523--542, 2016.

\bibitem{Hestenes1952}
M.~R. Hestenes and E.~Stiefel.
\newblock Methods of conjugate gradients for solving linear systems.
\newblock {\em J. Res. Natl. Bur. Stand.}, 49(6):409--436, 1952.

\bibitem{Nocedal2002}
J.~Nocedal, A.~Sartenaer, and C.~Zhu.
\newblock On the behavior of the gradient norm in the steepest descent method.
\newblock {\em Comput. Optim. Appl.}, 22(1):5--35, 2002.

\bibitem{Peaceman1955}
D.~W. Peaceman and H.~H. Rachford.
\newblock The numerical solution of parabolic and elliptic differential
  equations.
\newblock {\em J. Soc. Indust. Appl. Math.}, 3(1):28--41, 1955.

\bibitem{Pourbagher2018}
M.~Pourbagher and D.~K. Salkuyeh.
\newblock On the solution of a class of complex symmetric linear systems.
\newblock {\em Appl. Math. Lett.}, 76:14--20, 2018.

\bibitem{Raydan1993}
M.~Raydan.
\newblock On the {B}arzilai and {B}orwein choice of steplength for the gradient
  method.
\newblock {\em IMA J. Numer. Anal.}, 13(3):321--326, 1993.

\bibitem{Saad2003}
Y.~Saad.
\newblock {\em Iterative Methods for Sparse Linear Systems}.
\newblock SIAM, Philadelphia, 2nd edition, 2003.

\bibitem{Salkuyeh2015}
D.~K. Salkuyeh, D.~Hezari, and V.~Edalatpour.
\newblock Generalized successive overrelaxation iterative method for a class of
  complex symmetric linear system of equations.
\newblock {\em Int. J. Comput. Math.}, 92(4):802--815, 2015.

\bibitem{vandenDoel2012}
K.~{van den Doel} and U.~M. Ascher.
\newblock The chaotic nature of faster gradient descent methods.
\newblock {\em J. Sci. Comput.}, 51(3):560--581, 2012.

\bibitem{vanderSluis1986}
A.~{van der Sluis} and H.~A. {van der Vorst}.
\newblock The rate of convergence of conjugate gradients.
\newblock {\em Numer. Math.}, 48(5):543--560, 1986.

\bibitem{Wu2015}
S.-L. Wu.
\newblock Several variants of the {H}ermitian and skew-{H}ermitian splitting
  method for a class of complex symmetric linear systems.
\newblock {\em Numer. Linear Algebra Appl.}, 22(2):338--356, 2015.

\bibitem{Wu2017}
S.-L. Wu and C.-X. Li.
\newblock Modified complex-symmetric and skew-{H}ermitian splitting iteration
  method for a class of complex-symmetric indefinite linear systems.
\newblock {\em Numer. Algorithms}, 76(1):93--107, 2017.

\bibitem{Young1980}
D.~M. Young and K.~C. Jea.
\newblock Generalized conjugate-gradient acceleration of nonsymmetrizable
  iterative methods.
\newblock {\em Linear Algebra Appl.}, 34:159--194, 1980.

\bibitem{Yuan2006}
Y.-X. Yuan.
\newblock A new stepsize for the steepest descent method.
\newblock {\em J. Comput. Math.}, 24(2):149--156, 2006.

\end{thebibliography}
\bibliographystyle{abbrv}

\end{document}